\documentclass[12pt]{amsart}
\usepackage[usenames]{color}
     \makeatletter
     \def\section{\@startsection{section}{1}%
     \z@{.7\linespacing\@plus\linespacing}{.5\linespacing}%
     {\bfseries
     \centering
     }}
     \def\@secnumfont{\bfseries}
     \makeatother

\newtheorem{theorem}{Theorem}[section]
\newtheorem{lemma}[theorem]{Lemma}

\theoremstyle{definition}

\theoremstyle{remark}
\newtheorem{remark}[theorem]{Remark}
\numberwithin{equation}{section}
\newcommand{\be}{\begin{equation}}
\newcommand{\ee}{\end{equation}}
\newcommand{\bea}{\begin{eqnarray}}
\newcommand{\eea}{\end{eqnarray}}
\newcommand{\bean}{\begin{eqnarray*}}
\newcommand{\eean}{\end{eqnarray*}}
\newcommand{\brray}{\begin{array}}
\newcommand{\erray}{\end{array}}
\newcommand{\ben}{\begin{equation}{nonumber}}
\newcommand{\een}{\end{equation}{nonumber}}

\newtheorem{dfn}{Definition}[section]
\newtheorem{thm}[dfn]{Theorem}
\newtheorem{lema}[dfn]{Lemma}
\newtheorem{pro}[dfn]{Proposition}
\newtheorem{coro}[dfn]{Corollary}
\newtheorem{xmpl}[dfn]{Example}
\newtheorem{rmrk}[dfn]{Remark}

\newcommand{\bdfn}{\begin{dfn}}
\newcommand{\bthm}{\begin{thm}}
\newcommand{\blema}{\begin{lema}}
\newcommand{\bpro}{\begin{pro}}
\newcommand{\bcoro}{\begin{coro}}
\newcommand{\bxmpl}{\begin{xmpl}}
\newcommand{\brmrk}{\begin{rmrk}}

\newcommand{\edfn}{\end{dfn}}
\newcommand{\ethm}{\end{thm}}
\newcommand{\elema}{\end{lema}}
\newcommand{\epro}{\end{pro}}
\newcommand{\ecoro}{\end{coro}}
\newcommand{\exmpl}{\end{xmpl}}
\newcommand{\ermrk}{\end{rmrk}}

\newcommand{\half}{\frac{1}{2}}

\newcommand{\B}{\mathcal B}

\numberwithin{equation}{section}
\begin{document}

\title[Unitary Processes with Independent Increments]
{Characterization of Unitary Processes \\ \medskip with  Independent Increments}

\author[Un Cig Ji]{Un Cig Ji}
\thanks{  The first author is partially supported by the Korea Science and
Engineering Foundation (KOSEF) under a grant  by the Government of
Republic of  Korea  (MEST) (No. R01-2008-000-10843-0) while other
two authors  acknowledge  support by the UK- India Education and
Research  Initiative (UKIERI)  project  RA2007. The third author
would also like to thanks   CSIR, Government of India for partial
support   through Bhatnagar Fellowship.
 }
\address{Un Cig Ji: Department of Mathematics, Research Institute of Mathematical Finance,
Chungbuk National University, Cheongju 361-763, Korea}
\email{uncigji@chungbuk.ac.kr}

\author[Lingaraj Sahu]{Lingaraj Sahu}
\address{Lingaraj Sahu: Indian Institute of Science Education and Research (IISER) Mohali,
MGSIPAP Complex Sector 26, Chandigarh -16, India}
\email{lingaraj@iisermohali.ac.in}

\author[Kalyan B. Sinha]{Kalyan B. Sinha}
\address{Kalyan B. Sinha: Jawaharlal Nehru Centre for Advanced Scientific Research, Jakkur, Bangalore-64, India;
  Department of Mathematics, Indian Institute of Science, Bangalore-12, India}
\email{kbs\_jaya@yahoo.co.in}



\begin{abstract}
In this paper, we study unitary Gaussian processes with independent
increments with  which the unitary equivalence to a
Hudson-Parthasarathy evolution systems is proved. This  gives a
generalization of results in \cite{SSS} and \cite{SS} in the absence
of the stationarity condition.

\vskip1.5cm
\begin{center}
\textit{Dedicated to Robin L. Hudson on his 70th birthday}
\end{center}
\end{abstract}

\maketitle
\section{Introduction}

In the framework of the theory of quantum stochastic calculus
developed by the  work of Hudson and Parthasarathy \cite{hp1}, HP-
quantum stochastic differential equations (qsde)
 \be
\label{hpeqn00} dV_t=\sum_{\mu,\nu\ge 0}  V_t L_\nu^\mu(t)
\Lambda_\mu^\nu(dt),~ V_0=1_{\mathbf h \otimes \Gamma}, \ee (where
the coefficients $L_\nu^\mu(t)~: \mu,~\nu~\ge 0$ are operators in
the initial Hilbert space  $\mathbf h$ for almost every $t\ge 0$ and
$\Lambda_\mu^\nu$ are fundamental processes in the symmetric Fock
space  $\Gamma= \Gamma_{sym}(L^2(\mathbb R_+, \mathbf k))$ with
respect to  a fixed orthonormal basis  (in short `ONB') $\{E_j: j\ge
1\}$ of  the noise Hilbert space $\mathbf k$ ) have  been
formulated.  The conditions for existence and uniqueness  of a
solution $\{V_t\}$ were  studied by Hudson and Parthasarathy  and
many other authors. In particular when the coefficient operators
$\{L_\nu^\mu(t)~:t\ge 0, \mu,~\nu~\ge 0\}$ are locally  essentially
bounded in $\mathcal B(\mathbf h)$  satisfying unitarity  conditions
it is observed that the solution  $ \{V_t:t\ge 0\}$ is a unitary
process.

In particular, in the absence of the conservation martingale, the
equation  take the form $dV_t= \sum_{j}\{V_t L_j(t)
a^\dagger(dt)-V_t L_j^*(t) a(dt)\}+V_t G(t) dt$ with the unitary
condition $\sum_{j} L_j^* (t)  L_j(t)+2 Re~G(t)=0$  for almost every
$t\ge 0$ (Ref.\cite{gs,hp1}).

In a series of earlier work (Ref.\cite{SSS,SS}) it has been shown
that unitary evolution on $\mathbf h\otimes  \mathcal H$  with
stationary, independent increments and a Gaussian condition (where
$\mathbf h  $ and $\mathcal H$ are separable  Hilbert spaces) with
bounded  or  possibly unbounded generator  ( in the second case, one
needs some further conditions ) are unitarily isomorphic to solution
of qsde of the type (\ref{hpeqn00}) with time independent
coefficients.

In this  article  we are  interested in the characterization of
unitary evolutions with only   independent increments   on $\mathbf
h \otimes \mathcal H$  and with the assumption that the expectation
evolution relative to a distinguished  vector in $\mathcal H$  is
Lifshitz in the time  variable.

The  article is organized as follows: Section 2  is  meant  for
recalling  some preliminary ideas  and fixing  some  notations on
 linear  operators  on Hilbert spaces and  Section 3 collects some results associated with Hilbert  space and  properties  of evolutions.
 The main results of  section 3 are proved in the Appendix.
Section 3 also contain the description of  the  unitary  processes
with independent increments and the assumptions on them.
 Section  4 is dedicated to  the construction of a  Hilbert  space, called the  noise space   and  operator coefficients  associated with
 them.  The HP evolution system and its minimality are discussed in
 Section 5  and consequently  the unitary  equivalence of the solution
 with the unitary process is proven.

\section{Notation and Preliminaries}
\label{sec:notation and preliminaries}

We assume that all Hilbert spaces in this article are complex separable
with inner products which are anti-linear in the first variable.
For each Hilbert spaces $\mathcal{H}$ and $\mathcal{K}$
we denote the Banach spaces of all bounded linear operators from $\mathcal{H}$ to $\mathcal{K}$
and all trace class operators on $\mathcal{H}$ by $ \mathcal{B}(\mathcal{H},\mathcal{K})$ and
$\mathcal{B}_1(\mathcal{H})$, respectively, and the trace on $\mathcal{B}_1(\mathcal{H})$ by $\mathrm{Tr}(\cdot)$.
We note that for each $\xi\in\mathcal{H}\otimes\mathcal{K}$ and $h\in\mathcal{H}$,
there exists a unique vector $\langle\!\langle h,\xi\rangle\!\rangle$ in $\mathcal{K}$ such that
\begin{equation}\label{partinn}
\langle~\langle\!\langle h,\xi\rangle\!\rangle,~k\rangle
 =\langle\xi, h\otimes k\rangle,\forall k\in\mathcal K.
\end{equation}
In other words, $\langle\!\langle h,\xi \rangle\!\rangle=F_{h}^*\xi$,
where $F_{h}\in\mathcal{B}(\mathcal{K},\mathcal{H}\otimes\mathcal{K})$
is given by $F_{h}k=h\otimes k$.

Let $\mathbf{h}$ and $\mathcal{H}$ be two Hilbert spaces with orthonormal bases
$\{e_j:j\ge 1\}$ and $\{\zeta_j:j\ge 1\}$, respectively.
For each $A\in\mathcal{B}(\mathbf{h}\otimes\mathcal{H})$ and $u,v\in\mathbf{h}$
we define a linear operator $A(u,v)\in \mathcal{B}(\mathcal{H})$ by
\[
\langle \xi_1,A(u,v) \xi_2\rangle
 =\langle u\otimes \xi_1 ,A ~v\otimes \xi_2 \rangle,\forall \xi_1,\xi_2\in\mathcal{H}
\]
and read off the following properties (for the proof, see Lemma 2.1 in  \cite{SSS}):
\begin{lemma}\label{Auv}
Let $A,B\in\mathcal{B}(\mathbf{h}\otimes\mathcal{H})$. Then for any $u,v,u_i,v_i\in\mathbf{h}$ $(i=1,2)$ we have
\begin{itemize}
 \item  [(i)] $A(\cdot,\cdot):\mathbf{h}\times\mathbf{h}\mapsto\mathcal{B}(\mathcal{H})$
 is a separately continuous sesqui-linear map, and if
  $A(u,v)=B(u,v)$ for all $u,v\in\mathbf{h}$, then $A=B$,
 \item  [(ii)] $\|A(u,v)\|\le \|A\| \|u\| \|v\|$ and  $A(u,v)^*=A^*(v,u),$
 \item  [(iii)]  $A(u_1,v_1)B(u_2,v_2)=\left[A\left(|v_1><u_2|\otimes 1_\mathcal{H}\right)B\right](u_1,v_2),$
 \item  [(iv)] $AB(u,v)=\sum_{j \ge 1} A(u,e_j) B(e_j,v)$, where the series converges strongly,
 \item  [(v)]  $0\le A(u,v)^* A(u,v) \le \|u\|^2 A^*A(v,v)$,
 \item  [(vi)] for any $\xi_1,\xi_2\in\mathcal{H}$ we have
\begin{align*}
 \langle A(u_1,v_1)\xi_1,B(u_2,v_2)\xi_2\rangle
 &=\sum_{j\ge 1}\langle u_2\otimes \zeta_j,\left[B(|v_2><v_1|\otimes |\xi_2><\xi_1|)A^*\right]u_1\otimes \zeta_j\rangle\\
 &=\langle v_1\otimes \xi_1,\left[A^*(|u_1><u_2|\otimes 1_{\mathcal H})B\right]v_2 \otimes \xi_2 \rangle.
 \end{align*}
\end{itemize}
\end{lemma}

For each $A\in\mathcal{B}(\mathbf{h}\otimes\mathcal{H})$ and $\epsilon\in\mathbb{Z}_2=\{0,1\}$,
we define an operator $A^{(\epsilon)}\in\mathcal{B}(\mathbf{h}\otimes\mathcal{H})$ by
\[
A^{(\epsilon)}:=\left\{
                  \begin{array}{ll}
                    A & \hbox{if}~~ \epsilon=0, \\
                    A^* & \hbox{if}~~ \epsilon=1.
                  \end{array}
                \right.
\]
For $1\le k\le n,$ we define a unitary exchanging map
$P_{k,n}:\mathbf{h}^{\otimes n} \otimes \mathcal{H}\rightarrow\mathbf{h}^{\otimes n}\otimes\mathcal{H}$ by
\[
P_{k,n}(u_1\otimes\cdots\otimes u_n\otimes\xi)
 :=u_{\tau_{k,n}(1)}\otimes\cdots\otimes u_{\tau_{k,n}(n)}\otimes\xi
\]
on product vectors, where $\tau_{k,n}:=\left(k~k+1~\cdots~n\right)$
is a permutation on $\{1,2,\cdots,n\}$.
Let $\underline{\epsilon}=(\epsilon_1,\epsilon_2,\cdots,\epsilon_n)\in \mathbb{Z}_2^n$.
Consider the ampliation of the operator $A^{(\epsilon_k)}$ in $\mathcal B(\mathbf h^{\otimes n}
\otimes \mathcal H)$ given by
\[
A^{(n,\epsilon_k)}:=P_{k,n}^* (1_{\mathbf h^{\otimes n-1}} \otimes A^{(\epsilon_k)})P_{k,n}.
\]
Now we define the operator
\[
A^{(\underline{\epsilon})}
 :=\prod_{k=1}^n~A^{(n,\epsilon_k)}:=A^{(n,\epsilon_1)}\cdots A^{(n,\epsilon_n)}
\]
as in $\mathcal{B}(\mathbf{h}^{\otimes n} \otimes\mathcal{H})$.
Note that as here, through out  this article, the product symbol $\prod_{k=1}^n$
stands for  product with the ordering from $1$ to $n$.
For product vectors  $\underbar{u}, \underbar{v}\in\mathbf{h}^{\otimes n}$
one can see that
\begin{equation}\label{A-(epsilon)}
A^{(\underline{\epsilon})}(\underbar{u},\underbar{v})
 =\left(\prod_{i=1}^n A^{(n,\epsilon_i)}\right)(\underbar{u},\underbar{v})
 =\prod_{i=1}^n A^{(\epsilon_i)}(u_i,v_i)\in\mathcal{B}(\mathcal{H}),
\end{equation}
moreover, for $1\le m\le n$, we see that
\begin{equation}\label{Akn}
\left(\prod_{i=1}^m A^{(n,\epsilon_i)}\right)(\underbar{u},\underbar{v})
 =\prod_{i=1}^m A^{(\epsilon_i)}(u_i,v_i)   \prod_{i=m+1}^n \langle u_i,v_i\rangle
      \in  \mathcal B(\mathcal H).
\end{equation}
When $\underline{\epsilon}=\underline{0}\in\mathbb{Z}_2^n$,
for simplicity we shall write  $A^{(n,k)}$ for $A^{(n,\epsilon_k)}$
and $A^{(n)}$ for $A^{(\underline{\epsilon})}$.
%

\section {Unitary Processes with Independent Increments}

Let $\{U_{s,t}:0\le s\le t<\infty\}$ be a family of unitary
operators in $\mathcal{B}(\mathbf{h}\otimes\mathcal{H})$ with
$U_{s,s}=1$ for any $s\ge0$ and $\Omega $ be a fixed unit vector in $\mathcal{H}$.
Let us consider the family of unitary operators
$\{U_{s,t}^{(\epsilon)}\}$ in $ \mathcal{B}(\mathbf{h}\otimes \mathcal{H})$
for $\epsilon\in\mathbb{Z}_2$ given by
$ U_{s,t}^{(0)}=U_{s,t}$ and $U_{s,t}^{(1)}= U_{s,t}^*$.
As in Section \ref{sec:notation and preliminaries},
for fixed $n\ge 1$, $\underline{\epsilon}\in\mathbb{Z}_2^n$ and each $1\le k\le n$,
we define the families of operators $\{U_{s,t}^{(n,\epsilon_k)}\}$ and
$\{U_{s,t}^{(\underline{\epsilon})}\}$ in $\mathcal{B}(\mathbf{h}^{\otimes n}\otimes\mathcal{H})$.
By identity \eqref{A-(epsilon)}, for product vectors $\underbar{u},\underbar{v}\in \mathbf{h}^{\otimes n}$
and $\underline{\epsilon}\in\mathbb{Z}_2^n$, we have
\[
U_{s,t}^{(\underline{\epsilon})}(\underbar{u},\underbar{v})
=\prod_{i=1}^n U_{s,t}^{(\epsilon_i)}(u_i,v_i).
\]
Furthermore, for $\underbar{s}=(s_1,s_2,\cdots,s_n)$,
$\underbar{t}=(t_1,t_2,\cdots,t_n)$ such that
$0\le s_1\le t_1\le s_2\le\ldots \le s_n\le t_n<\infty$, we define
$U_{\underbar{s},\underbar{t}}^{(\underline{\epsilon})}\in\mathcal{B}(\mathbf{h}^{\otimes n}\otimes\mathcal{H})$
by setting
\begin{equation}\label{U-underbar-st}
U_{\underbar{s},\underbar{t}}^{(\underline{\epsilon})}:=\prod_{k=1}^n
U_{s_k,t_k}^{(n,\epsilon_k)}.
\end{equation}
Then for $\underbar{u}=\otimes_{k=1}^n u_k,
\underbar{v}=\otimes_{k=1}^n v_k\in \mathbf h^{\otimes n}$ we have
\[
U_{\underbar{s},\underbar{t}}^{(\underline{\epsilon})}(\underbar{u},\underbar{v})
 =\prod_{k=1}^n U_{s_k,t_k}^{(\epsilon_k)}(u_k,v_k).
\]
When $\underline{\epsilon}=\underline{0}$, we write
$U_{\underbar{s},\underbar{t}}$ for $U_{\underbar{s},\underbar{t}}^{(\underline{\epsilon})}$.
For $\alpha,\beta\ge0$ and $\underbar{s}=(s_1,s_2, \cdots, s_n)$,
$\underbar{t}=(t_1,t_2, \cdots, t_n)$, we write $\alpha\le\underbar{s},\underbar{t}\le \beta$
if $ \alpha\le s_1\le t_1\le s_2 \le \ldots \le s_n\le  t_n\le \beta$.

\bigskip
We assume the following on the family of unitary
$\{U_{s,t}\in\mathcal{B}(\mathbf{h}\otimes\mathcal{H})\}$.

\bigskip
\noindent
\textbf{Assumption A:}
\begin{itemize}
  \item[\textbf{(A1)}] \textbf{(Evolution}) for any $0\le r\le s\le t<\infty, ~
  U_{r,s}U_{s,t}=U_{r,t}$  and $U_{s,s}=1,$
  \item[\textbf{(A2)}] \textbf{(Independence of increments)} for any $0\le s_i\le t_i<\infty$ ($i=1,2$)
                       such that $[s_1,t_1)\cap[s_2,t_2)=\emptyset$,
\begin{itemize}
  \item [(i)] $U_{s_1,t_1}(u_1,v_1)$ commutes with $U_{s_2,t_2}(u_2,v_2)$ and $U_{s_2,t_2}^*(u_2,v_2) $
              for any $u_i,v_i \in \mathbf{h}$ ($i=1,2$).
  \item [(ii)] for any $s_1\le \underbar{q},\underbar{r}\le t_1$, $s_2\le\underbar{s},\underbar{t}\le t_2$
               and $\underbar{u},\underbar{v}\in\mathbf h^{\otimes n}$, $\underbar{w},\underbar{z}\in\mathbf{h}^{\otimes m}$
               and $\underline{\epsilon}\in \mathbb{Z}_2^n$, $\underline{\epsilon}^\prime \in \mathbb{Z}_2^m$,
\[
\langle \Omega,U_{\underbar{q},\underbar{r}}^{(\underline{\epsilon})}(\underbar{u},\underbar{v})
  U_{\underbar{s},\underbar{t}}^{(\underline{\epsilon}^\prime)}(\underbar{p},\underbar{w})\Omega\rangle
=\langle\Omega,U_{\underbar{q},\underbar{r}}^{(\underline{\epsilon})}(\underbar{u},\underbar{v})\Omega\rangle
 \langle\Omega,U_{\underbar{s},\underbar{t}}^{(\underline{\epsilon}^\prime)}(\underbar{p},\underbar{w})\Omega\rangle.
\]
\end{itemize}
\end{itemize}

\medskip
\noindent \textbf{Assumption B:~ (Regularity)}\enspace for any
$\infty
>t\ge s\ge0$,
\[
\sup\left\{\left|\left\langle\Omega,(U_{s,t}-1)(u,v)\Omega\right\rangle\right|\,:\,
\|u\|=\|v\|=1\right\}
 \le  C  |t-s|\]  for some  positive  constant   $C$ independent of  $s,t.$

\begin{remark}
Unlike \cite{SSS, SS}, in the \textbf{Assumption  A}, the stationarity condition is not assumed.
\end{remark}

As in \cite{SSS,SS}, we need further assumptions for Gaussianity and minimality:

\medskip
\noindent \textbf{Assumption C:~ (Gaussianity)}\enspace for each
$t\ge s\ge0$ and any $u_k,v_k\in \mathbf h$, $\epsilon_k\in
\mathbb{Z}_2$ ($k=1,2,3$),
\begin{equation}\label{eqn:Gaussian condition}
\lim_{t \downarrow s}\frac{1}{t-s}
 \left\langle\Omega,\left(\prod_{k=1}^3 (U_{s,t}^{(\epsilon_k)}-1)(u_k,v_k)\right)\Omega\right\rangle=0.
\end{equation}

\medskip
\noindent \textbf{Assumption D:~ (Minimality)}\enspace the set
\[
 \mathcal{S}_0
 =\left\{U_{\underbar{s},\underbar{t}}(\underbar{u},\underbar{v})\Omega\,:\,
 \begin{array}{l}
   \underbar{s}=(s_1,s_2, \cdots, s_n),~\underbar{t}=(t_1,t_2,\cdots, t_n)~\textrm{with}~0\le \underbar{s},\underbar{t}<\infty,\\
   \underbar{u}=\otimes_{k=1}^nu_k,\underbar{v}=\otimes_{k=1}^n v_k\in \mathbf{h},~~ n\ge1
 \end{array}
 \right\}
\]
is total in $\mathcal{H}$.
%
%

\begin{remark}
The \textbf{Assumption D} is not really a restriction, one can as
well work by  replacing $\mathcal{H}$ by the closure of the linear
span of $\mathcal{S}_0$.
\end{remark}
%

\subsection{Vacuum Expectation}

Let us look at the various evolutions associated with the
$\{U_{s,t}\}.$ Define a two parameter family of operators
$\{T_{s,t}\}$ on $\mathbf h $ by
\[
\left\langle u,T_{s,t}v\right\rangle
 :=\left\langle \Omega, U_{s,t}(u,v)\Omega\right\rangle,\qquad\forall u,v \in\mathbf{h}.
\]
For each $t\ge s\ge 0$, since $U_{s,t}$ is unitary, $T_{s,t}$ is a
contractions.

\brmrk The  \textbf{Assumption B} implies  $\| T_{s,t}-1\|\le C
|t-s|.$ In particular \\
$\lim_{t\downarrow s}  T_{s,t}=1$  uniformly
in $s.$\ermrk

\begin{lemma}
Under the \textbf{Assumptions A }  and {\bf  B,} the family
$\{T_{s,t}\}$ of contractions satisfies
\begin{itemize}
  \item [(i)] for any $r\le s\le t<\infty$, $T_{r,s}
  T_{s,t}=T_{r,t}$  and $T_{s,s}=1_{\mathbf h}$
 \item [(ii)]  for any   $  t'\ge t\ge s \ge 0,  \|
 T_{s,t^\prime}-T_{s,t}\|\le C|t'-t|.$
\end{itemize}
\end{lemma}

\begin{proof}
(i)  The evolution and independent increment property of
$\{U_{s,t}\}$ and the  definition of $T_{s,t}$ gives the result.\\
 (ii) By (i), for a fixed  $s\ge 0$ and any $t^\prime\ge
t\ge s$, we have
\[
\| T_{s,t^\prime}-T_{s,t}\|=\|
T_{s,t}\left(T_{t,t^\prime}-1\right)\|\le\|
T_{s,t}\|\|T_{t,t^\prime}-1\| \le C|t'-t|.
\]
Therefore, by \textbf{Assumption B}, we have $\lim_{t^\prime
\downarrow t}\| T_{s,t^\prime}-T_{s,t}\|=0$.
\end{proof}
Let us note down  the  following  result  about  the  generator of
the evolutions of the type    $T_{s,t}$ which  is  proved in the
Appendix.

\bthm \label{Ttt}
   There exists  a Lebesgue measurable  function  $G:\mathbb R_+\rightarrow  \B(\mathbf h)$  such that  G is  locally essentially
 bounded  and   $$T_{s,t}-1=\int_s^t T_{s, \tau}  G(\tau)  d \tau.$$
 \ethm

\begin{equation}\label{eqn:assumption T}
\lim_{h\downarrow0}\frac{T_{t,t+h}-I}{h}=G(t)
\end{equation}
in the operator norm  topology  for almost  every  $t$.

We shall need  the following  observation (see Equation (6.2) in
\cite{SSS}):
\begin{equation}\label{UTknc0}
\sum_{k\ge 1} \left\|\left(U_{s,t}-1\right)(\phi_k,w)\Omega\right\|^2
 =\left\langle w,\left(1-T_{s,t}\right)w\right\rangle
  +\left\langle\left(1-T_{s,t}\right)w,w\right\rangle
\end{equation}
for any $w\in\mathbf{h}$, where $\{\phi_k\}$ is an complete orthonormal basis of $\mathbf{h}$.

\begin{lemma}\label{4Ut-111}
Under the \textbf{Assumptions  C},  for each $s\ge0$, we have the
following:
\begin{itemize}
  \item [(i)] for any $n\ge 3$, $\underbar{u},\underbar{v} \in \mathbf{h}^{\otimes n}$ and
$\underline{\epsilon}\in \mathbb{Z}_2^n$,
\begin{equation}
\lim_{t\downarrow s}\frac{1}{t-s}
 \left\langle\Omega,\left(\prod_{k=1}^n\left[\left(U_{s,t}^{(\epsilon_k)}-1\right)(u_k,v_k)\right]\right)\Omega\right\rangle=0,
\end{equation}

  \item [(ii)] for any vectors $u,v\in \mathbf{h}$,
   product vectors $ \underbar{w},\underbar{z}\in \mathbf{h}^{\otimes n}$ and
$\epsilon\in\mathbb{Z}_2$, $\underline{\epsilon^\prime}\in\mathbb{Z}_2^n$,
\begin{align}
\lim_{t\downarrow s}&\frac{1}{t-s} \left\langle\left(U_{s,t}-1\right)^{(\epsilon)}(u,v)\Omega,
    \left(U_{s,t}^{(\underline{\epsilon^\prime})}-1\right)(\underbar{p},\underbar{w})\Omega\right\rangle
    \label{UtUt*inner}\\
&= (-1)^\epsilon
\lim_{t\downarrow s}\frac{1}{t-s} \left\langle\left(U_{s,t}-1\right)(u,v)\Omega,
     \left(U_{s,t}^{(\underline{\epsilon^\prime})}-1\right)(\underbar{p},\underbar{w})\Omega\right\rangle.
    \nonumber
\end{align}
\end{itemize}
\end{lemma}

\begin{proof}
(i) The  proof is a simple modification of the proof of Lemma 6.6 in \cite{SSS}.

(ii)  The idea here is similar to that in the   proof  of Lemma 6.7
in \cite{SSS}. For $\epsilon =0$, it is obvious. To see this for
$\epsilon=1$, put
\[
\Phi=\left(U_{s,t}^{(\underline{\epsilon^\prime})}-1\right)(\underbar{p},\underbar{w})
\]
and consider the following
\begin{align}
\lim_{t\downarrow s}&\frac{1}{t-s}\left\langle\left(U_{s,t}+U_{s,t}^*-2\right)(u,v)\Omega,\Phi\Omega\right\rangle
   \label{**}\\
&=-\lim_{t\downarrow s}\frac{1}{t-s}\left\langle\left[\left(U_{s,t}^*-1\right)\left(U_{s,t}-1\right)\right](u,v)\Omega,
  \Phi\Omega\right\rangle
   \nonumber\\
&=-\lim_{t\downarrow s}\frac{1}{t-s}\sum_{k\ge 1} \left\langle\left(U_{s,t}-1\right)(e_k,v)\Omega,
  \left(U_{s,t}-1\right)(e_k,u)\Phi\Omega\right\rangle.
   \nonumber
\end{align}
On the other hand, we have
\begin{align*}
&\left|
\frac{1}{t-s}
 \sum_{k\ge 1}\left\langle\left(U_{s,t}-1\right)(e_k,v)\Omega,
  \left(U_{s,t}-1\right)(e_k,u)\Phi\Omega\right\rangle\right|^2\\
&\quad\le
\left(\sum_{k\ge 1} \frac{1}{t-s}\left\|\left(U_{s,t}-1\right)(e_k,v)\Omega\right\|^2\right)
\left(\sum_{k\ge 1} \frac{1}{t-s}\left\|\left(U_{s,t}-1\right)(e_k,u)\Phi\Omega\right\|^2\right).
\end{align*}
By \eqref{UTknc0} and (iv) in Lemma, the above quantity is equal to
\begin{align*}
& 2 Re\left\langle v,\frac{1-T_{s,t}}{t-s}v\right\rangle
  \frac{1}{t-s}\left\langle\Phi\Omega,\left[\left(U_{s,t}^*-1\right)\left(U_{s,t}-1\right)\right](u,u)\Phi\Omega\right\rangle\\
&=2 Re \left\langle v, \frac{1-T_{s,t}}{t-s} v\right\rangle
  \frac{1}{t-s}\left\langle\Phi\Omega,\left(2-U_{s,t}^*-U_{s,t}\right)(u,u)\Phi\Omega\right\rangle,
\end{align*}
and by \eqref{eqn:assumption T}, $\lim_{t\downarrow s}\left\langle
v, \frac{1-T_{s,t}}{t-s} v\right\rangle=\left\langle v,
G(s)v\right\rangle$ for any $v\in\mathbf h$. Also, by
\textbf{Assumption C} we have
\[
\lim_{t\downarrow s}\frac{1}{t-s}\left\langle\Phi\Omega,\left(2-U_{s,t}^*-U_{s,t}\right)(u,u)\Phi\Omega\right\rangle=0.
\]
Therefore, we have
\[
\lim_{t\downarrow s}\frac{1}{t-s}
 \sum_{k\ge 1} \left\langle\left(U_{s,t}-1\right)(e_k,u)\Omega,
  \left(U_{s,t}-1\right)(e_k,v)\left(U_{s,t}^{(\underline{\epsilon^\prime})}-1\right)
      (\underbar{p},\underbar{w})\Omega\right\rangle
=0,
\]
which, by applying \eqref{**}, implies (\ref{UtUt*inner}).
  \end{proof}

For each $s\ge0$ and for  vectors  $u,v,p,w\in\mathbf{h}$ the
identity (\ref{UtUt*inner}) gives
\begin{align}
\lim_{t\downarrow s}&\frac{1}{t-s}\left\langle\left(U_{s,t}-1\right)^{(\epsilon)}(u,v)\Omega,
 \left(U_{s,t}-1\right)^{(\epsilon^\prime)}(p,w)\Omega\right\rangle
  \label{UtUt*10}\\
 &=(-1)^{\epsilon+\epsilon^\prime}\lim_{t\downarrow s}\frac{1}{t-s}
  \left\langle\left(U_{s,t}-1\right)(u,v)\Omega,\left(U_{s,t}-1\right)(p,w)\Omega\right\rangle.
  \nonumber
\end{align}

We now introduce the partial trace $\mathrm{Tr}_\mathcal{H}$ which is a linear map from
 $\mathcal{B}_1(\mathbf{h}\otimes\mathcal{H})$  to $\mathcal{B}_1(\mathbf{h})$ defined  by
\[
\langle u,\mathrm{Tr}_\mathcal{H}(B) v\rangle
:=\sum_{j\ge 1} \langle u\otimes\zeta_j , B  v\otimes\zeta_j\rangle,
 \qquad \forall u,v\in \mathbf h
\]
for $B\in\mathcal{B}_1(\mathbf{h}\otimes\mathcal{H})$.
In particular, $\mathrm{Tr}_\mathcal{H}(B)=\mathrm{Tr}(B_2)\,B_1$ for $B=B_1\otimes B_2$.
Then we define a family of operators $\{Z_{s,t}\}_{0\le s\le t}$ on the Banach space $\mathcal{B}_1(\mathbf{h})$ by
\begin{equation}\label{eqn:Z-{s,t}}
Z_{s,t}(\rho)
=\mathrm{Tr}_\mathcal{H}\left[U_{s,t}\left(\rho \otimes |\Omega><\Omega|\right)U_{s,t}^*\right],
\qquad \rho \in \mathcal{B}_1(\mathbf{h}).
\end{equation}
Thus, for any $u,v,p,w\in \mathbf{h}$, we have
\begin{equation}\label{Ztsimple}
\left\langle p,Z_{s,t}(|w><v|) u \right\rangle :=\left\langle
U_{s,t}(u,v)\Omega,U_{s,t}(p,w)\Omega\right\rangle.
\end{equation}

 For
$\rho\in\mathcal{B}_1(\mathbf{h})$, by the definition of $Z_{s,t}$
and trace norm (see page no. 47 in \cite{gelshi}), we have
\begin{align*}
\|Z_{s,t}(\rho)\|_1
 &=\left\|\mathrm{Tr}_\mathcal{H}[ U_{s,t}\left(\rho\otimes|\Omega><\Omega|\right)U_{s,t}^*]\right\|_1\\
 &=\sup_{\phi,\psi:~ \mathrm{ons}~ \mathrm{of}~ \mathbf{h}}
   \sum_{k\ge 1}
   \left|\left\langle\phi_k ,\mathrm{Tr}_{\mathcal H}
   \left[U_{s,t}\left(\rho\otimes|\Omega><\Omega|\right)U_{s,t}^*\right]\psi_k\right\rangle\right|\\
 &\le\sup_{\phi,\psi:~ \mathrm{ons}~ \mathrm{of}~ \mathbf{h}}
  \sum_{j,k\ge 1}\left|\left\langle\phi_k\otimes\zeta_j,
   U_{s,t}\left(\rho\otimes|\Omega><\Omega|\right)U_{s,t}^*\psi_k\otimes\zeta_j\right\rangle\right|\\
 &\le\left\|U_{s,t}\left(\rho\otimes|\Omega><\Omega|\right)U_{s,t}^*\right\|_1\le \|\rho\|_1.
\end{align*}
Thus $Z_{s,t}$ is contractive. For any $u,v\in \mathbf h, \|
U_{s,t}(u,v)\Omega\|^2=\langle u, Z_{s,t} (|v><v| )u\rangle$ and
positivity of $Z_{s,t}$ is  clear.

\begin{lemma}
Under the \textbf{Assumptions A} and \textbf{B}, $\{Z_{s,t}\}$ is a
family of positive contractive  map on $\mathcal{B}_1(\mathbf{h})$
satisfying
\begin{itemize}
  \item [(i)] for any  $0\le r\le s\le t<\infty$,
  $Z_{r,s}Z_{s,t}=Z_{r,t}, Z_{s,s}=1$
  \item [(ii)] for any   $  t'\ge t\ge s \ge 0 ,
  \|Z_{s,t^\prime}-Z_{s,t}\|_1\le 4C|t'-t|,$
  \item [(iii)] For any  $\rho\in \mathcal B_1 (\mathbf h),  Tr(Z_{s,t}\rho)=Tr(\rho).$
\end{itemize}
\end{lemma}

\begin{proof}

(i)To prove evolution property of  $Z_{s,t}$  it is enough to show
that for any $u,v,p,w\in \mathbf h, \langle
U_{s,t}(u,v)\Omega,U_{s,t}(p,w)\Omega\rangle=\langle p, Z_{r,t}
(|w><v| )u\rangle= \langle p,  Z_{r,s}  Z_{s,t} (|w><v|) u\rangle. $
This can be checked  by using  evolution and independent increment
property of unitary
family  $U_{s,t}.$\\
 (ii)  For any rank one operator $\rho=|w><v|$,
$w,v\in\mathbf{h}$, we have

\begin{align*}
\|(Z_{s,t}-1) (|w><v|)\|_1
&=\sup_{\{\phi\},\{\psi\}~ ons~ of~ \mathbf h}\sum_{k\ge 1}|\langle \phi_k ,(Z_{s,t}-1)(|w><v|)\psi_k\rangle|\\
&=\sup_{\phi,\psi}\sum_{k\ge 1}|\langle U_{s,t}(\psi_k,v)\Omega,
 U_{s,t}(\phi_k,w)\Omega\rangle-\overline{\langle\psi_k,v\rangle}\langle\phi_k,w\rangle|\\
&\le \sup_{\phi,\psi}\sum_{k\ge 1}|\langle
(U_{s,t}-1)(\psi_k,v)\Omega, ( U_{s,t}-1) (\phi_k,w)\Omega\rangle|\\
&+\sup_{\phi,\psi}\sum_{k\ge 1}|\overline{\langle\psi_k,v\rangle}\langle \Omega, ( U_{s,t}-1) (\phi_k,w)\Omega|\\
&+\sup_{\phi,\psi}\sum_{k\ge 1}|\overline{\langle\Omega,
(U_{s,t}-1)(\psi_k,v)\Omega\rangle}\langle\phi_k,w\rangle|\\
& \le
\sup_{\phi,\psi}\left[\sum_{k\ge1}\|(U_{s,t}-1)(\psi_k,v)\Omega\|^2\right]^{1/2}
\left[\sum_{k\ge1}\|(U_{s,t}-1)(\psi_k,w)\Omega\|^2\right]^{1/2}\\
&\quad+\sup_{\phi,\psi}\left[\sum_{k\ge
1}|\langle\psi_k,v\rangle|^2\right]^{1/2} \left[\sum_{k\ge
1}|\langle  \phi_k,(T_{s,t}-1)w\rangle|^2\right]^{1/2}\\
&\quad+\sup_{\phi,\psi}\left[\sum_{k\ge
1}|\langle\phi_k,w\rangle|^2\right]^{1/2} \left[\sum_{k\ge
1}|\langle  \psi_k,(T_{s,t}-1)v\rangle|^2\right]^{1/2}.
\end{align*}

Hence by identity (\ref{UTknc0}) and {\bf Assumption  B} we obtain
\begin{align*}
&\|(Z_{s,t}-1) (|w><v|)\|_1 \\
&\le 2 \|(T_{s,t}-1)\|  \|w\| \|v\|+\|(T_{s,t}-1)w\|
\|v\|+\|(T_{s,t}-1)v\| \|w\|
\\
&\le 4  C |t-s|   \|w\| ~\|v\|.
\end{align*}
Now any for  $\rho=\sum_{k}\lambda_k|\phi_k><\psi_k| \in  \mathcal
B_1(\mathbf h )$, where $\{\phi_k\}$ and $\{\psi_k\}$ are two
orthonormal bases of $\mathbf{h}$ and we have
\begin{align*}
\|Z_{s,t}(\rho)- \rho\|_1 \le 4 C
\left(\sum_{k}\left|\lambda_k\right|\right)  |t-s| \le 4 C
\|\rho\|_1|t-s|
\end{align*}
and hence  \be \label{Zst-1}   \|Z_{s,t}- 1\|_1  \le 4 C |t-s|.\ee
By evolution property  and  contractivity of  $\{Z_{s,t}\}$
\[
\| Z_{s,t^\prime}-Z_{s,t}\|=\|
Z_{s,t}\left(Z_{t,t^\prime}-1\right)\|\le\|
Z_{s,t}\|\|Z_{t,t^\prime}-1\| \le 4 C|t'-t|.
\]
(iii)  It can be proved as in lemma   6.5  in \cite{SSS}
\end{proof}
The theorem \ref{StG} in the Appendix leads to following result for
 $Z_{s,t}.$ \bthm \label{Ztt}
    Under the {\bf Assumption  A, B} there exists  a Lebesgue measurable  function  $\mathcal L:\mathbb R_+\rightarrow   \mathcal B(\B_1(\mathbf h))$  such that  $\mathcal L$ is  locally essentially
 bounded in $\mathbb R_+$ and  such that  $$Z_{s,t}-1=\int_s^t Z_{s, \tau}
  \mathcal L(\tau)  d \tau,~~~~~~~~~~~~~~~~~~~~~~~~ \lim_{h\downarrow0}\frac{Z_{t,t+h}-I}{h}=\mathcal
  L(t).$$
 \ethm

\section{Construction of Noise Space}

Put $M_0:=\{(\underbar{u},\underbar{v},\underline{\epsilon}):
\underbar{u}=\otimes_{i=1}^n u_i, \underbar{v}=\otimes_{i=1}^n
v_i\in \mathbf{h}^{\otimes n},
\underline{\epsilon}=(\epsilon_1,\cdots,\epsilon_n)\in
\mathbb{Z}_2^n, ~ n\ge 1\}$. Now, consider the relation $``\sim"$ on
$M_0$ as defined in \cite{SSS} :
$(\underbar{u},\underbar{v},\underline{\epsilon})\sim
(\underbar{p},\underbar{w},\underline{\epsilon}^\prime)$ if
$\underline{\epsilon}=\underline{\epsilon}^\prime$ and
 $|\underbar{u}><\underbar{v}|=|\underbar{w}>< \underbar{z}|\in
 \B(\mathbf h^{\otimes n}).$  Now consider the   algebra $M$
generated by $M_0/\sim$ with multiplication structure given by
$(\underbar{u},\underbar{v},\underline{\epsilon})\cdot
(\underbar{p},\underbar{w},\underline{\epsilon}^\prime)=
(\underbar{u} \otimes \underbar{w},\underbar{v} \otimes
\underbar{z},\underline{\epsilon}\oplus
\underline{\epsilon}^\prime)$. For each $s\ge 0$ we define a scalar
valued map $K_s$ on $M\times M$ by setting, for
$(\underbar{u},\underbar{v},\underline{\epsilon}),
(\underbar{p},\underbar{w},\underline{\epsilon}^\prime) \in M_0$,
\[
K_s\left((\underbar{u},\underbar{v},\underline{\epsilon}), (\underbar{w},
   \underbar{z},\underline{\epsilon}^\prime)\right)
:=\lim_{t\downarrow s}\frac{1}{t-s}
 \left\langle\left(U_{s,t}^{(\underline{\epsilon})}-1\right)(\underbar{u},\underbar{v})\Omega,
      \left(U_{s,t}^{\underline{\epsilon}^\prime}-1\right)(\underbar{p},\underbar{w})\Omega\right\rangle
\]
if the limit exists.

\begin{theorem}\label{noise} For  almost every  $s$
\begin{itemize}
  \item [(i)] the map $K_s$ is a positive definite kernel on $M$,
  \item [(ii)] there exists a unique (up to unitary equivalence) separable Hilbert space $\mathbf{k}_s$,
               an embedding $\eta_s:M\rightarrow\mathbf{k}_s$ such that
\begin{equation}\label{eta-dense11}
\left\{\eta_s(\underbar{u},\underbar{v},\underline{\epsilon}):
(\underbar{u},\underbar{v},\underline{\epsilon})\in M_0\right\}~\mbox{ is total in}~\mathbf{k}_s,
\end{equation}
\begin{equation}\label{eta-kelnel-11}
\left\langle \eta_s(\underbar{u},\underbar{v},\underline{\epsilon}),
       \eta_s(\underbar{p},\underbar{w},\underline{\epsilon}^\prime)\right\rangle
=K_s\left((\underbar{u},\underbar{v},\underline{\epsilon}),(\underbar{p},\underbar{w},\underline{\epsilon}^\prime)\right),
\end{equation}
  \item [(iii)] for any $(\underbar{u},\underbar{v},\underline{\epsilon})\in M_0$,
$\underbar{u}=\otimes_{i=1}^n u_i$, $\underbar{v}=\otimes_{i=1}^n v_i$
and $\underline{\epsilon}=(\epsilon_1,\cdots,\epsilon_n)$
\begin{equation}\label{eta-n11}
\eta_s(\underbar{u},\underbar{v},\underline{\epsilon})
 =\sum_{i=1}^n\prod_{k\ne i} \left\langle u_k,v_k \right\rangle \eta_s(u_i,v_i,\epsilon_i),
\end{equation}
  \item [(iv)] $\eta_s(u,v,1)=-\eta_s(u,v,0)$ for any $u,v\in\mathbf{h}$,
  \item [(v)] for fixed $u,v,p,w\in\mathbf{h}$,
the map $s\mapsto K_s((u,v),(p,w))=\left\langle
\eta_s(u,v),\eta_s(p,w)\right\rangle$ is  Lebesgue measurable   and
 essentially  bounded in $\mathbb R_+$.
\end{itemize}
\end{theorem}

\begin{proof}
(i) The proof is exactly same as the proof of Lemma 7.1 in
\cite{SSS}. By Lemma \ref{4Ut-111}, for elements
$(\underbar{u},\underbar{v},\underline{\epsilon}),
(\underbar{p},\underbar{w},\underline{\epsilon}^\prime) \in M_0$,
$\underline{\epsilon}\in \mathbb{Z}_2^m$  and
$\underline{\epsilon}^\prime\in \mathbb{Z}_2^n$, we have
\begin{align}
&K_s\left((\underbar{u},\underbar{v},\underline{\epsilon}),
   (\underbar{p},\underbar{w},\underline{\epsilon}^\prime)\right) \label{KK}\\
&=\lim_{t \downarrow s}\frac{1}{t-s}
 \left\langle\left(U_{s,t}^{(\underline{\epsilon})}-1\right)(\underbar{u},\underbar{v})\Omega,
   \left(U_{s,t}^{(\underline{\epsilon}^\prime)}-1\right)(\underbar{p},\underbar{w})\Omega\right\rangle
  {\nonumber} \\
&=\sum_{1\le i\le m,~1\le j\le n}~\prod_{k\ne i}\overline{\langle u_k,v_k \rangle}
    \prod_{l\ne j} \langle p_l,w_l \rangle
    \nonumber\\
&\times\lim_{t\downarrow s}\frac{1}{t-s}
  \left\langle\left(U_{s,t}-1\right)^{(\epsilon_i)}(u_i,v_i)\Omega,
       \left(U_{s,t}-1\right)^{(\epsilon_j^\prime)}(p_j,w_j)\Omega\right\rangle.
    \nonumber
\end{align}
Since
\begin{align*}
&\left\langle\left(U_{s,t}-1\right)(u,v)\Omega,\left(U_{s,t}-1\right)(p,w)\Omega\right\rangle\\
&\quad=\left\langle U_{s,t}(u,v)\Omega,U_{s,t}(p,w)\Omega\right\rangle-\overline{\langle u,v\rangle} \langle p,w\rangle\\
&\qquad -\overline{\langle u,v\rangle}
\left\langle\Omega,\left(U_{s,t}-1\right)(p,w)\Omega\right\rangle
        -\overline{\left\langle\Omega,\left(U_{s,t}-1\right)(u,v)\Omega\right\rangle} \langle  p,w\rangle\\
&\quad=\left\langle p,\left(Z_{s,t}-1\right)(|w><v|)u\right\rangle
 -\overline{\langle u,v\rangle} \left\langle p,\left(T_{s,t}-1\right)w\right\rangle
 -\overline{\left\langle u,\left(T_{s,t}-1\right)v\right\rangle}\langle p,w\rangle,
\end{align*}
the existence of the limits on the right hand side of (\ref{KK})
follows from the identity (\ref{UtUt*inner}) and theorems \ref{Ttt}
  and \ref{Ztt} and   $K_s$ is given by
\begin{align}
& K_s((u,v,\epsilon ),(p,w,\epsilon^\prime))
    \label{kernel11}\\
&=(-1)^{\epsilon+\epsilon^\prime}
 \lim_{t\downarrow s}\left\{\left\langle p,\frac{Z_{s,t}-1}{t-s}(|w><v|)u\right\rangle
 - \overline{\langle u,v\rangle }\left\langle p,\frac{T_{s,t}-1}{t-s}w\right\rangle\right\}
    \nonumber\\
&\qquad-(-1)^{\epsilon+\epsilon^\prime}
 \lim_{t\downarrow s}\overline{\left\langle u, \frac{T_{s,t}-1}{t-s}v\right\rangle} \langle p,w\rangle
   \nonumber\\
 &= (-1)^{\epsilon+\epsilon^\prime}\left\{\left\langle p,\mathcal{L}(s)(|w><v|)u\right\rangle
  -\overline{\langle u,v\rangle}\left\langle p,G(s) w\right\rangle
  -\overline{\left\langle u, G(s) v\right\rangle}\langle p,w\rangle\right\}.
   \nonumber
\end{align}
This expression can be  extend  to the algebra $M$ by
sesqui-linearly.

(ii) For each $s\ge 0$, the Kolmogorov's construction \cite{krp} to
the pair $(M,K_s)$ provides a Hilbert space $\mathbf{k}_s$ as the
closure of the span of
$\left\{\eta_s(\underbar{u},\underbar{v},\underline
{\epsilon}):(\underbar{u},\underbar{v},\underline{\epsilon})\in
M\right\}$.

(iii) Again as in \cite{SSS}, for any
$(\underbar{p},\underbar{w},\underline{\epsilon}^\prime)\in M_0$, by
Lemma \ref{4Ut-111}, we have
\begin{align*}
\left\langle\eta_s(\underbar{u},\underbar{v},\underline{\epsilon}),
      \eta_s(\underbar{p},\underbar{w},\underline{\epsilon}^\prime)\right\rangle
&=K_s\left((\underbar{u},\underbar{v},\underline{\epsilon}),
      (\underbar{p},\underbar{w},\underline{\epsilon}^\prime)\right)\\
&=\sum_{i=1}^n\prod_{k\ne i} \overline{\langle u_k,v_k\rangle}
  \left\langle\eta_s(u_i,v_i,\epsilon_i),\eta_s(\underbar{p},\underbar{w},\underline{\epsilon}^\prime)\right\rangle.
\end{align*}
Since
$\left\{\eta_s(\underbar{p},\underbar{w},\underline{\epsilon}^\prime)
 :(\underbar{p},\underbar{w},\underline{\epsilon}^\prime)\in M_0\right\}$ is a total
subset of $\mathbf{k}_s$,  (\ref{eta-n11}) follows.

(iv) By (\ref{UtUt*inner}), we have
\[
\left\langle\eta_s(u,v,1),\eta_s(\underbar{p},\underbar{w},\underline{\epsilon}^\prime)\right\rangle
=\left\langle-\eta_s(u,v,0),\eta_s(\underbar{p},\underbar{w},\underline{\epsilon}^\prime)\right\rangle
\] and hence  $\eta_s(u,v,1)=-\eta_s(u,v,0).$
\\
By parts (iii)  and (iv)  of this theorem, it is clear that
$\mathbf{k}_s$  is spanned by  the family $\{\eta_s(u,v): u,v\in
\mathbf h\},$  where we have written $\eta_t(u,v)$ for
$\eta_t(u,v,0)$.\\

 Since $ G(s), \mathcal
L(s)$ are essentially bounded in norm it follow from
\eqref{kernel11}  that $\eta_s(.,.):\mathbf h \times \mathbf h
\rightarrow \mathbf k_s$ is continuous  and thus  separability of
$\mathbf k_s$  follows from that of $\mathbf h.$\\
(v)  Since $\mathcal
L(s)$ and $G(s)$ are measurable essentially bounded,  result follows
from the identity \eqref{kernel11}.
\end{proof}

For any two  orthonormal  bases  $\{\phi_k\}, \{\psi_k\}$ of
$\mathbf
 h, $  the collection   of vectors $\{\eta_s(\phi_k,\psi_l):k,l\ge 1\}$
 is a countable  total family  in $\mathbf k_s$  and  $s\mapsto  \langle \eta_s(u,v), \eta_s(p,w)
 \rangle=K_s((u,v),(p,w))$ is a Lebesgue measurable  function.  Thus
 $s\mapsto  \langle \eta_s(u,v)$  is measurable.
The family  $\{\mathbf k_s:s\ge 0\}$  spanned by $\{\eta_s(u,v):s\ge
0, u,v\in \mathbf h\},$ is a measurable field of Hilbert spaces 
\cite{dix}.

 For any $T\ge 0,$ define
$K^T((u,v),(p,w))=\int_0^T K_s((u,v),(p,w))ds$
\[=\int_0^T \{\left\langle p,\mathcal{L}(s)(|w><v|)u\right\rangle
  -\overline{\langle u,v\rangle}\left\langle p,G(s) w\right\rangle
  -\overline{\left\langle u, G(s) v\right\rangle}\langle
  p,w\rangle\} ds.\]  Since each  $K_s$ is positive  definite it can be seen that   $K^T$  is a positive
definite  kernel. Let the associated  Hilbert  space  $\mathbf k^T.$
  There exists a family of
vectors  $\eta^T(u,v)$ which spans the Hilbert  space  $\mathbf k^T$
such that

$$\langle \eta^T(u,v),\eta^T(p,w)\rangle =K^T((u,v),(p,w))$$
$$ =\int_0^T
K_s((u,v),(p,w))ds=\int_0^T \langle \eta_s(u,v),\eta_s(p,w)\rangle
ds$$

In $\mathbf k^T$ there exists a bounded self adjoint operator  $A$
with  absolutely  continuous  simple spectrum   such that
$A\eta^T(u,v)(s)=s \eta_s(u,v)$ for almost every  $s\in [0,T]$  and
 $\mathbf k^T$
is the direct integral  $\int_{[0,T]}^{\oplus}\mathbf{k}_s ds$ (Ref
\cite{dix}).  There is natural isometric  embedding of $\mathbf
k^{T}$  in  $\mathbf k^{T'}$  for  $T\le T'$  by setting
$\eta_s^{T,T'}(u,v)=\eta_s^T(u,v)$  for all $0\le s\le  T$  and  $0$
for $s\in (T,T'].$


 \brmrk The integral
$ \int_{\mathbb R_+} K_s((u,v),(u,v))ds=\int_{\mathbb R_+}
\|\eta_s(u,v)\|^2 d$s need not exist and therefore
$\int_{\mathbb{R}_+}^{\oplus}\mathbf{k}_s ds$  may not  be defined.
\ermrk

\begin{lemma}\label{H,Lj11}
Under the hypothesis of Theorem \ref{noise}, we have the following:
\begin{itemize}
 \item[(i)] There exists a unique strong measurable  family of
 bounded
 operators  $L(t):\mathbf h\rightarrow   \mathbf h  \otimes  \mathbf k_t$
such that
\[\|L(t) v\|^2 =-2\mathrm{Re}\left\langle
v,G(t)v\right\rangle,\qquad \forall v\in \mathbf{h}.
\]

\item [(ii)] The map $t\mapsto L(t)$ is  essentially norm bounded.

\end{itemize}
\end{lemma}

\begin{proof}
(i) By the identity (\ref{kernel11}), for any $u,v\in\mathbf{h}$, we
have for almost every $t\ge 0$
\begin{align}
\|\eta_t(u,v)\|^2 &=\left\langle u,\mathcal{L}(t)
(|v><v|)u\right\rangle
  -\overline{\langle u,v\rangle}\left\langle u,G(t)v\right\rangle
  -\overline{\left\langle u, G(t)v\right\rangle}\langle u,v\rangle
   \label{etanorm}
\end{align}
and thus
\begin{align*}
 & \sum_{k }\| e_k\otimes \eta_t(e_k,v)\|^2=\sum_{k}\|\eta_t(e_k,v)\|^2\\
 &=\sum_{k}\left[\left\langle e_k,\mathcal{L}(t)(|v><v|)e_k\right\rangle
   -\overline{\left\langle e_k,v\right\rangle}\left\langle e_k,G(t)v\right\rangle
   -\overline{\left\langle e_k,G(t)v\right\rangle}\left\langle e_k,v\right\rangle\right]\\
&=\mathrm{Tr}\left(\mathcal{L}(t)(|v><v|)\right)
   -\left\langle v,G(t)v\right\rangle
   -\overline{\left\langle v,G(t)v\right\rangle}.
\end{align*}
Moreover, since $Z_{s,t}$ is trace preserving it follows that $
\mathrm{Tr}\left(\mathcal{L}(t)(|v><v|)\right)=0. $  Therefore $
   \sum_{k }\| e_k\otimes \eta_t(e_k,v)\|^2=
   -2  Re\left\langle v,G(t)v\right\rangle.$ This implies  that  $\sum_{k } e_k\otimes
\eta_t(e_k,v)$   is
 convergent  in norm and  in fact for almost every  $t$  it  defines  a  bounded operator
  $L(t):\mathbf h\rightarrow   \mathbf h  \otimes  \mathbf k_t$  given by
  $L(t)v=\sum_{k } e_k\otimes
  \eta_t(e_k,v)$  with \begin{equation}\label{lj*lj11}
\|L(t)v\|^2
 =-2Re\left\langle v,G(t)v\right\rangle.
\end{equation}  The strong
measurability of   $t\mapsto  L(t)$ follows from the definition.
\\
The part (ii)  follows from the essential  norm boundedness  of
$G(.).$

\end{proof}
%

\section{Hudson-Parthasarathy (HP) Evolution Systems and Equivalence}

\subsection{HP Evolution Systems}\label{subsec:HP evolution aystems}
 In order to  simplify the discussion of the existence  and
uniqueness of the solution of HP type  quantum stochastic
differential equation in  $\Gamma_{sym}(\int_{\mathbb R_+} ^\oplus
\mathbf k_s ds)$ and to be able  to refer to  existing literature,
it is convenient to introduce the following point of view which
allow us to embed the process in the standard Fock space
$\Gamma=\Gamma_{sym}(L^2(\mathbb R_+, \mathbf k))$  where $\mathbf
k=l^2(\mathbb{N}).$

Note that for almost every  $t\ge 0$, $\mathbf{k}_t$ is a complex
separable Hilbert space. Setting $d(t)=$ the dimension of
$\mathbf{k}_t$, $d:\mathbb{R}_+\rightarrow\mathbb{N} \cup
\{\infty\}$ is measurable and defining $ \Lambda_n=\{t:d(t)=n\}$, $\mathbb{R}_+$
can be written as disjoint union $\bigcup_{n=1}^\infty \Lambda_n$ of
measurable sets. Let us consider the Hilbert space $l^2(\mathbb{N})$
with a fixed orthonormal basis $\{E_j:j\ge 0\}$. Now for $t\in
\Lambda_n$, $n<\infty$ we embed $\mathbf{k}_t$ as the $n$
dimensional subspace $Span\{E_j:1\le j\le n\}$ of $\mathbf{k}$ and
for $t\in \Lambda_\infty$, $\mathbf{k}_t$ identified with
$\mathbf{k}$. Then the direct integral
$\int_{\mathbb{R}_+}^{\oplus}\mathbf{k}_t dt=\bigoplus _{n\ge 1}
L^2(\Lambda_n, \mathbb C^{n})\subseteq
L^2(\mathbb{R}_+,\mathbf{k})$. If $\Lambda_{\infty} = \emptyset$ , then  $\int_{\mathbb{R}_+}^{\oplus}\mathbf{k}_t dt$
  is isomprphic to $L^2(\mathbb R_+,\mathbb C^n) $ for some $n.$ 

For any subset $\mathbf{D}\subseteq  L^2(\mathbb R_+, \mathbf k)$,
let $\mathcal{E}(\mathbf{D})$ be the subspace of $\Gamma$ which is
spanned by the set $\left\{\textbf{e}(f):f\in \mathbf{D}\right\}$ of
exponential vectors defined as:
\[
\textbf{e}(f):=\oplus_{n\ge 0}\frac{f^{\otimes n}}{\sqrt{n!}}.
\]

For $0\le s<t<\infty$ and
$f\in\mathcal{K}=L^2(\mathbb{R}_+,\mathbf{k})$, we denote the
functions $1_{[0,s]}f$, $1_{(s,t]}f$ and $1_{[t,\infty)}f$ by
$f_{s]}$, $f_{(s,t]}$ and $f_{[t}$, where $1_A$ is the indicator
function of $A\subset[0,\infty)$. Then the Hilbert spaces
$\mathcal{K}$ and $\Gamma$ can be decomposed as
$\mathcal{K}=\mathcal{K}_{s]}\oplus\mathcal{K}_{[s,t)}\oplus\mathcal{K}_{[t}$
and $\Gamma=\Gamma_{s]}\otimes\Gamma_{[s,t)}\otimes\Gamma_{[t}$ via
the unitary isomorphism given by:
\[
\Gamma\ni\textbf{e}(f)\quad \longleftrightarrow \quad
\textbf{e}(f_{s]})\otimes\textbf{e}(f_{(s,t]})\otimes \textbf{e}(f_{[t})
\in\Gamma_{s]}\otimes\Gamma_{[s,t)}\otimes\Gamma_{[t},
\]
where $\mathcal{K}_{s]}=L^2([0,s), \mathbf k)$,
$\mathcal{K}_{[s,t)}=L^2([s,t), \mathbf k)$,
$\mathcal{K}_{[t}=L^2([t, \infty), \mathbf k)$ and
$\Gamma_{s]}=\Gamma(\mathcal{K}_{s]})$,
$\Gamma_{[s,t)}=\Gamma(\mathcal{K}_{[s,t)})$,
$\Gamma_{[t}=\Gamma(\mathcal{K}_{[t})$.

Let us consider the Hudson-Parthasarathy (HP) type equation on $\mathbf{h}\otimes\Gamma$:
\begin{equation}\label{hpeqn}
V_{s,t}
 =1_{\mathbf{h}\otimes\Gamma}
   +\sum_{\mu,\nu\ge 0}\int_s^t V_{s,\tau} L_\nu^\mu(\tau)\Lambda_\mu^\nu(d\tau).
\end{equation}
Here the coefficients $L_\nu^\mu(\tau)$ ($\mu,\nu\ge 0$) are operators in $\mathbf{h}$  and
$\Lambda_\mu^\nu(t)$ are fundamental processes define by
\begin{equation} \label{Lmunu}
\Lambda_\mu^\nu(t)
=\left\{\begin{array}{lll}
 & t1_{\mathbf{h}\otimes\Gamma} & \mbox{for}~ (\mu,\nu)=(0,0), \\
 & a\left(1_{[0,t]}\otimes E_j(t)\right) &  \mbox{for}~ (\mu,\nu)=(j,0),\\
 & a^\dag\left(1_{[0,t]}\otimes E_k (t)\right) & \mbox{for}~ (\mu,\nu)=(0,k),\\
 & \Lambda\left(1_{[0,t]}\otimes |E_k(t)><E_j(t)|\right) & \mbox{for}~ (\mu,\nu)=(j,k),
 \end{array}\right.
\end{equation}
where $E_j(t)=E_j$ for $j\in\{1,2,\cdots d(t)\}$ and $E_j(t)=0$
otherwise. With respect to the orthonormal basis $E_j(t)$  we have
bounded operators $\{L_j(t):t\ge 0,j\ge 1\}$ in $\mathbf{h}$
 such that \be \label{ljjt} \langle u,L_j(t) v \rangle=\langle E_j,\eta_t(u,v) \rangle, \forall
 u,v\in\mathbf{h}.\ee
For detail  about quantum stochastic calculus see \cite{krp, gs}).
%


Now, let us state the main result of this article.
\begin{theorem}\label{mainthm}
Under Assumptions \textbf{A, B, C} and \textbf{D}, we have the following.
\begin{itemize}
  \item [(i)] The HP type equation
\begin{equation}\label{hpeqn,st11}
V_{s,t}
 =1_{\mathbf{h}\otimes \Gamma}
  +\sum_{\mu,\nu\ge 0}\int_s^t V_{s,r} L_\nu^\mu(r) \Lambda_\mu^\nu(dr)
\end{equation}
on $\mathbf{h}\otimes\Gamma_{\mathrm{sym}}(\mathcal{K})$ with coefficients $L_\nu^\mu(t)$ given by
\begin{equation} \label{hpcoefi11}
L_\nu^\mu(t)=\left\{ \begin{array} {lll}
 & G(t) & \mbox{for}~ (\mu,\nu)=(0,0),\\
 &L_j(t) & \mbox{for}~ (\mu,\nu)=(j,0),\\
 & -L_k(t)^*& \mbox{for}~ (\mu,\nu)=(0,k),\\
 & 0 & \mbox{for}~ (\mu,\nu)=(j,k)
 \end{array}
 \right.
  \end{equation}
admit a unique unitary solution $V_{s,t}$.
  \item [(ii)] There exists a unitary isomorphism
$\widetilde{\Xi}:\mathbf{h}\otimes\mathcal{H}\rightarrow \mathbf{h}\otimes\Gamma$ such that
\begin{equation}\label{U=V}
U_{s,t}=\widetilde{\Xi}^*V_{s,t}\widetilde{\Xi},\qquad \forall\,0\le s\le t<\infty.
\end{equation}
\end{itemize}
\end{theorem}

Here we shall sketch the proof of   part (i) of the theorem and
postpone that of part  (ii) to the next two sub sections. For
$\underline{\epsilon}=(\epsilon_1,\epsilon_2,\cdots,
\epsilon_n)\in\mathbb{Z}_2^n$, we define
$V_{s,t}^{(\underline{\epsilon})}\in\mathcal{B}(\mathbf{h}^{\otimes
n}\otimes\Gamma)$ by setting
$V_{s,t}^{(\epsilon)}\in\mathcal{B}(\mathbf{h}\otimes\Gamma)$ by
\begin{equation*}
V_{s,t}^{(\epsilon)}
 =\left\{
          \begin{array}{ll}
              V_{s,t} & \hbox{for } \epsilon=0,\\
              V_{s,t}^* & \hbox{for } \epsilon=1.
            \end{array}
         \right.
\end{equation*}

The next result verifies the properties of \textbf{Assumption  A} for the family
$V_{s,t}$ with $\Omega=\textbf{e}(0)\in\Gamma$.

\begin{lemma}\label{Vstbasic1}
The unitary solution $\{ V_{s,t}\}$ of HP equation \eqref{hpeqn,st11} satisfies
\begin{itemize}
  \item[(i)] for any $0\le r\le  s \le t<\infty$, $V_{r,t}=V_{r,s}V_{s,t}$,
\item[(ii)] for $[q,r)\cap [s,t)=\emptyset$, $V_{q,r}(u,v)$ commute with
            $V_{s,t}(p,w)$ and $V_{s,t}(p,w)^*$ for any $u,v,p,w \in \mathbf h$,
\item[(iii)] for any $0\le s\le t<\infty$,
\[
\left\langle\textbf{e}(0),V_{s,t}(u,v)\textbf{e}(0)\right\rangle
=\left\langle u,T_{s,t}v\right\rangle,\qquad \forall\,
u,v\in\mathbf{h}.
\]
\end{itemize}
\end{lemma}

\begin{proof}
(i) For fixed $0\le r\le s\le t<\infty$, we set $W_{r,t}=V_{r,s} V_{s,t}$.
Then by \eqref{hpeqn,st11}, we have
\begin{align*}
W_{r,t}
&=V_{r,s}+\sum_{\mu,\nu\ge 0}\int_s^t V_{r,s}V_{s,q} L_\nu^\mu (\tau) \Lambda_\mu^\nu(d\tau)\\
&=W_{r,s}+\sum_{\mu,\nu\ge 0}\int_s^t   W_{r,q}
L_\nu^\mu(\tau)\Lambda_\mu^\nu(d\tau),
\end{align*}
where $W_{r,s}=V_{r,s}V_{s,s}=V_{r,s}$.
Thus the family $\{ W_{r,t}\}$ of unitary operators also satisfies the HP equation
(\ref{hpeqn,st11}).  Hence by uniqueness of the solution of this quantum
stochastic differential equation, $W_{r,t}=V_{r,t}$ for any $0\le r\le s\le t<\infty$,
and the result follows.

(ii) For any $0\le s\le t<\infty$, $V_{s,t}\in
\mathcal{B}(\mathbf{h}\otimes \Gamma_{[s,t]})$. $p,w\in\mathbf{h}$,
$V_{s,t}p,w\in\mathcal{B}(\Gamma_{[s,t]})$ and the statement
follows.

(iii) Let us define \[ \left\langle
u,\widetilde{T}_{s,t}v\right\rangle:=
\left\langle\textbf{e}(0),V_{s,t}(u,v)\textbf{e}(0)\right\rangle
,\qquad \forall\, u,v\in\mathbf{h}.
\]Then $\widetilde{T}_{s,t}$  is  a contractive  family of operators  and
by the cocycle property of $V_{s,t},$ \be \label{t-t}
\widetilde{T}_{s,t} =1+\int_s^t \widetilde{T}_{s,\tau}G(\tau) d\tau.
\ee  Thus $\widetilde{T}_{s,t} -T_{s,t}$  satisfies the differential
equation
\[ \widetilde{T}_{s,t} -T_{s,t}= \int_s^t
(\widetilde{T}_{s,\tau}-T_{s,\tau})G(\tau) d\tau.
\]  Since $G(\tau)  $  is an essentially norm  bounded  function, an
iteration  of (\ref{t-t}) will leads to  $\widetilde{T}_{s,t}
=T_{s,t}$ for all $s,t.$
\end{proof}
\noindent Consider the family of operators $\widetilde{Z}_{s,t}$
defined by
\[
\widetilde{Z}_{s,t}(\rho)
=\mathrm{Tr}_{\mathcal{H}}
  \left[V_{s,t}(\rho \otimes|\textbf{e}(0)><\textbf{e}(0)|)V_{s,t}^*\right],
\qquad \forall\, \rho \in\mathcal{B}_1(\mathbf{h}).
\]
As for $Z_{s,t}$, it can be seen that $\widetilde{Z}_{s,t}$ is a
contractive family of maps  on $\B_1(\mathbf h)$ and, in particular,
for any $u,v,p,w\in\mathbf{h}$,
\[
\left\langle p, \widetilde{Z}_{s,t}(|w><v|)u\right\rangle
 =\left\langle V_{s,t} (u,v)\textbf{e}(0),V_{s,t}(p,w)\textbf{e}(0)\right\rangle.
\]

\begin{lemma}
The family $\{\widetilde{Z}_{s,t}\}$ is a uniformly continuous evolution
of contraction on $\mathcal{B}_1(\mathbf{h})$ and $\widetilde{Z}_{s,t}=Z_{s,t}$,
where $Z_{s,t}$ is given as in \eqref{eqn:Z-{s,t}}.
\end{lemma}

\begin{proof}
By \eqref{hpeqn,st11} and Ito's formula, for $u,v,p,w\in\mathbf{h}$
\begin{align*}
&\left\langle
p,\left[\widetilde{Z}_{s,t}-1\right](|w><v|)u\right\rangle
=\left\langle
V_{s,t}(u,v)\textbf{e}(0),V_{s,t}(p,w)\textbf{e}(0)\right\rangle
   -\overline{\left\langle u,v\right\rangle}\left\langle p,w\right\rangle\\
&\quad=\int_s^t \left\langle
V_{s,\tau}(u,v)\textbf{e}(0),V_{s,\tau}(p,G(\tau)
z)\textbf{e}(0)\right\rangle d\tau
  +\int_s^t \left\langle V_{s,\tau}(u,G(\tau) v)\textbf{e}(0),V_{s,\tau}(p,w)\textbf{e}(0)\right\rangle d\tau\\
&\qquad\qquad+\int_s^t
  \left\langle V_{s,\tau}(u,L_j(\tau)v)\textbf{e}(0),V_{s,\tau}(p,L_j(\tau)z)\textbf{e}(0)\right\rangle d\tau\\
&\quad=\int_s^t \left\langle p,\widetilde{Z}_{s,\tau}(|G(\tau)
w><v|)u\right\rangle d\tau
  +\int_s^t\left\langle p,\widetilde{Z}_{s,\tau} (|w><G(\tau) v|)u\right\rangle d\tau\\
&\qquad\qquad+\sum_{j\ge 1}\int_s^t\left\langle
p,\widetilde{Z}_{s,\tau}(|L_j(\tau)w><L_j(\tau)v|)u\right\rangle
d\tau.
\end{align*}
Thus by identity   (\ref{ljjt}) for $\{L_j(t)\},$   we have that
\begin{equation}\label{Ztilda}
\left\langle
p,\left[\widetilde{Z}_{s,t}-1\right](\rho)u\right\rangle
 =\int_s^t \left\langle p,\widetilde{Z}_{s,\tau}\mathcal{L}(\tau)(\rho)u\right\rangle d\tau,
\end{equation}
where $\rho=|w><v|$. Thus the family $\{\widetilde{Z}_{s,t}\}$
satisfies the differential equation
\[
\widetilde{Z}_{s,t}(\rho) =\rho+\int_s^t
\widetilde{Z}_{s,\tau}\mathcal{L}(\tau)(\rho)d\tau,\qquad
\rho\in\mathcal{B}_1(\mathbf{h}).
\]
Therefore, proceeding as in the proof of Lemma \ref{Vstbasic1} (iii)
we can conclude that $\widetilde{Z}_{s,t}=Z_{s,t}$.
\end{proof}
%

\subsection{Minimality of HP Evolution Systems}

In this section we shall show the minimality of the HP evolution system $\{V_{s,t}\}$
discussed in Section \ref{subsec:HP evolution aystems}
which will be needed to prove (ii) in Theorem \ref{mainthm},
i.e., to establish unitary equivalence of $U_{s,t}$ and $V_{s,t}$.
We shall prove here that the subset
\[
 \mathcal{S}^\prime
 =\left\{V_{\underbar{s},\underbar{t}}(\underbar{u},\underbar{v})\textbf{e}(0)\,:\,
 \begin{array}{l}
   \underbar{s}=(s_1,s_2, \cdots, s_n),~\underbar{t}=(t_1,t_2,\cdots, t_n)~\textrm{with}~0\le \underbar{s},\underbar{t}<\infty,\\
   \underbar{u}=\otimes_{i=1}^nu_i,\underbar{v}=\otimes_{i=1}^n v_i\in \mathbf{h},~~ n\ge1
 \end{array}
 \right\}
\]
is total in the symmetric Fock space $\Gamma(\mathcal{K})\subseteq \Gamma(L^2(\mathbb{R}_+,\mathbf{k}))$,
where
\[
V_{\underbar{s},\underbar{t}}(\underbar{u},\underbar{v})\textbf{e}(0)
:=V_{s_1,t_1}(u_1,v_1)\cdots V_{s_n,t_n}(u_n,v_n)\textbf{e}(0).
\]

Let $\tau\ge0$ be fixed and as in (Ref. \cite{SSS}), we note that for any $0\le s < t\le \tau$,
$u,v\in \mathbf{h}$,
\begin{align}\label{vst-1}
\frac{1}{t-s}\left[V_{s,t}-1\right](u,v)\textbf{e}(0)
 =\gamma(s,t,u,v)+\rho(s,t,u,v)+\zeta(s,t,u,v)+\varsigma(s,t,u,v),
\end{align}
where these vectors in the Fock space $\Gamma$ are given by
\begin{align*}
 \gamma(s,t,u,v)
   &:=\frac{1}{t-s}\sum_{j\ge 1} \int_s^t \left\langle u,L_j(\lambda)v\right\rangle a_j^\dag(d \lambda)\,\textbf{e}(0),\\
   \rho(s,t,u,v)&:=\frac{1}{t-s}\int_s^t  \left\langle
u,G(\lambda)v\right\rangle
d\lambda~~\textbf{e}(0),\\
\zeta(s,t,u,v)
   &:=\frac{1}{t-s}\sum_{j\ge 1}\int_s^t\left(V_{s,\lambda}-1\right)\left(u,L_j(\lambda)v\right)a_j^\dag(d\lambda)\,\textbf{e}(0),\\
\varsigma(s,t,u,v)
   &:=\frac{1}{t-s}\int_s^t\left(V_{s,\lambda}-1\right)\left(u,G(\lambda) v\right)d\lambda\,\textbf{e}(0).
\end{align*}
Note that any $\phi\in\Gamma$ can be written as
$\phi=\phi^{(0)}\oplus\phi^{(1)}\oplus\cdots$, where $\phi^{(n)}$ is
in the $n$-fold symmetric tensor product
$L^2(\mathbb{R}_+,\mathbf{k})^{\otimes_s n}\equiv L^2(\Sigma_n)
\otimes \mathbf{k}^{\otimes n}$. Here $\Sigma_n$ is the $n$-simplex
$\{\underbar{t}=(t_1,t_2,\cdots, t_n):0\le t_1 <t_2\cdots
<t_n<\infty\}$.

\blema \label{qs-esti} Let $ u,v\in \mathbf h$  and let $ C_{\tau}=
4 e^\tau \sup\{
  \|L(\lambda) \|^2 +  \|G(\lambda)\|^2:0\le   \lambda \le  \tau\}.$
  Then  for any $
0\le s\le t\le \tau,$

{ (i)} \be \|(V_{s,t}-1) v \textbf{e}(0)\|^2\le  C_{\tau} |t-s|
\|v\|^2.\ee

{(ii)} For any $u\in \mathbf h$ \bean &&\|\sum_{j\ge 1} \int_s^t
V_{s,\lambda} (u,L_j(\lambda)v) a_j^\dag (d\lambda) \textbf{e}(0)\|^2\\
&&\le \| u\|^2 \| \sum_{j\ge 1}  \int_s^t V_{s,\lambda} L_j(\lambda)
~d\lambda ~~ v \otimes  \textbf{e}(0)\|^2.\eean \elema
\begin{proof}{(i)} By
estimates of   quantum stochastic integration  (Proposition 27.1,
\cite{krp}) \bean
&&\|(V_{s,t}-1) v \textbf{e}(0)\|^2\\
&&=\| \sum_{j\ge 1}\int_s^t V_{s,\lambda} L_j(\lambda) a_j^\dag
(d\lambda) ~~v \textbf{e}(0)+ \int_s^t V_{s,\lambda}G(\lambda) d
\lambda ~~v \textbf{e}(0)\|^2\\
&&\le 2 e^\tau \int_s^t\{  \sum_{j\ge 1}\|L_j  (\lambda)v\|^2
+\|G(\lambda)v\|^2\} d
\lambda \\
&&\le 2 e^\tau\|v\|^2  \int_s^t\{  \|L (\lambda)\|^2
+\|G(\lambda)\|^2\} d
\lambda \\
 && =  \|v\|^2  C_{\tau} |t-s|.\eean
 {(ii)} For  any  $\phi$ in the Fock space $\Gamma(L^2(\mathbb R_+, \mathbf k))$,
\bean && \langle \phi,\sum_{j\ge 1} \int_s^t V_{s,\lambda}
(u,L_j(\lambda) v)
a_j^\dag (d\lambda) \textbf{e}(0)\rangle |^2\\
&&=|\langle u\otimes\phi,\{ \sum_{j\ge 1} \int_s^t
V_{s,\lambda}L_j(\lambda)
a_j^\dag (d\lambda) \}v  \textbf{e}(0)\rangle |^2\\
&&\le \| u\otimes\phi\|^2 \|\{ \sum_{j\ge 1} \int_s^t
V_{s,\lambda}L_j(\lambda) a_j^\dag (d\lambda)\} v  \textbf{e}(0)
\|^2. \eean Since $\phi$ is arbitrary, the first  inequality
follows. Thus further by the estimates of quantum stochastic
integration
\[\|\sum_{j\ge 1} \int_s^t
V_{s,\lambda} (u,L_j(\lambda)v) a_j^\dag (d\lambda)
\textbf{e}(0)\|^2 \le 2 e^\tau \| u\|^2   \int_s^t  \sum_{j\ge
1}\|V_{s,\lambda}  L_j(\lambda)v \|^2 d \lambda\] \be \label{vla+}
\le 2 e^\tau \| u\|^2
\int_s^t \|L(\lambda)v \|^2\\
d \lambda\\
 \le |t-s| \|u\|^2 \|v\|^2  C_{\tau}. \ee

\end{proof}
 \blema \label{lemma-vst-1} Let  $C_{\tau}^\prime =4 e^{2 \tau}\sup\{
  \|(L(\alpha )\otimes  1) L(\lambda) \|^2 +  \|(G(\alpha)\otimes  1)~L(\lambda) \|^2   :  \alpha, \lambda \in [0,\tau]\} $  and
   $C_{\tau}^{\prime \prime} =4 e^{2 \tau}\sup\{
  \|(L(\alpha )\otimes  1) G(\lambda) \|^2 +  \|(G(\alpha)\otimes  1)~G(\lambda) \|^2   :  \alpha, \lambda \in [0,\tau]\}. $ Then for any $ u,v\in
\mathbf h, 0\le s\le t\le \tau$
\begin{itemize}
\item[{ (i)}] $\|(V_{s,t}-1)(u,v)~~\textbf{e}(0)\|^2\le C_{\tau}\|u\|^2
\|v\|^2 |t-s|.$\\

\item[{ (ii)}] $\sup\{ \|\zeta(s,t,u,v) \|^2:0\le s\le t \le \tau\} \le  C_\tau^\prime  \|u\|^2
  \| v\|^2 $
and\\
 $\|\varsigma(s,t,u,v)\|\le  \sqrt{ C_{\tau}^{\prime \prime} |t-s|}  \|u\|
  \| v\|.$\\

\item[{(iii)}]For any $\phi\in \Gamma(L^2(\mathbb R_+, \mathbf k)),$
~ $\lim_{t \downarrow s}\langle \phi,  \zeta(s,t,u,v)\rangle =0$ and
\[\lim_{t \downarrow s} \langle \phi, \gamma(s,t,u,v)\rangle =\sum_{j\ge 1}
\langle u,L_j(s)v\rangle \overline{\phi^{(1)}_j} (s)= \langle
\phi^{(1)} (s), \eta_s(u,v)\rangle,~~\mbox{a.e.}~~ s\ge 0.\]
\end{itemize}
\elema

\begin{proof} { (i)}  By    (\ref{hpeqn,st11}) and (\ref{vla+}) we have
   \bean
  && \|(V_{s,t}-1)(u,v)~~\textbf{e}(0)\|^2 \\
   &&=\|\sum_{j\ge 1} \int_s^t
  V_{s,\lambda}(u,L_j(t) v) a_j^\dag(d\lambda)~\textbf{e}(0)+ \int_s^t
  V_{s,\lambda}(u, G(\lambda) v)~\textbf{e}(0) d\lambda\|^2\\
  && \le 2 \|\sum_{j\ge 1} \int_s^t
  V_{s,\lambda}(u,L_j(\lambda) v) a_j^\dag(d\lambda)~~\textbf{e}(0)\|^2+[\int_s^t
  \|V_{s,\lambda}(u, G(\lambda) v)~~\textbf{e}(0)\| d\lambda]^2\\
   && \le 4 e^\tau \|u\|^2  \|v\|^2  \int_s^t [
  \|L(\lambda)\|^2 +  \|G(\lambda)\|^2 ] d \lambda \\
     && \le C_{\tau} \|u\|^2 \|v\|^2 |t-s|.
  \eean

\noindent { (ii) }  By  inequalities (\ref{vla+})  we have
  \bean
  && \|\zeta(s,t,u,v)\|^2=\frac{1}{|t-s|^2}\|\sum_{j\ge 1} \int_s^t
  (V_{s,\lambda}-1)(u,L_j(\lambda)v) a_j^\dag (d\lambda) ~~\textbf{e}(0)\|^2\\
    && \le  \frac{ 2 e^\tau \|u\|^2}{|t-s|^2}\int_s^t  \sum_{j\ge 1} \|
 (V_{s,\lambda}-1)L_j(\lambda)v ~~\textbf{e}(0)\|^2 d\lambda.
 \eean
   Now as in  Lemma \ref{qs-esti} (i),  the above quantity can be
   estimated by
 \bean
     && \le  \frac{4 e^{2\tau} \|u\|^2}{|t-s|^2}\int_s^t\sum_{j\ge 1}
     \{ \int_s^\lambda \sum_{i\ge 1}
  \|L_i(\alpha ) L_j(\lambda) v\|^2 +   \|G(\alpha)~L_j(\lambda) ~v\|]^2 \} d \alpha ~d \lambda\\
   && \le  \frac{4 e^{2\tau} \|u\|^2}{|t-s|^2}\int_s^t
      \int_s^\lambda\{
  \|(L(\alpha )\otimes  1) L(\lambda) v\|^2 +  \|(G(\alpha)\otimes  1)~L(\lambda) ~v\|^2 \} d \alpha~ d \lambda,
  \eean  which leads to the  statement.

\noindent  Also we have
  \bean
  && \|\varsigma(s,t,u,v)\|=\frac{1}{|t-s|}\| \int_s^t
  (V_{s,\lambda}-1)
(u,G (\lambda)v) d \lambda ~~\textbf{e}(0)\|\\
    && \le  \frac{1}{|t-s|} \int_s^t
  \|(V_{s,\lambda}-1)
(u,G(\lambda)v)~~\textbf{e}(0)\| d\lambda.
  \eean
Thus, similarly  as above,  the estimate follows.

\noindent  { (iii)} For any $f\in L^2(\mathbb R_+,\mathbf k).$
 Let us consider
 \bean
 &&\langle
\textbf{e}(f),\zeta(s,t,u,v)\rangle=\langle
\textbf{e}(f),\frac{1}{t-s} \sum_{j\ge 1}\int_s^t (V_{s,\lambda}-1)
(u,L_j(\lambda)v)
a_j^\dag (d\lambda) ~~ \textbf{e}(0)\rangle\\
&&=\frac{1}{t-s} \sum_{j\ge 1}\int_s^t \overline{f_j(\lambda)}
\langle \textbf{e}(f),(V_{s,\lambda}-1) (u,L_j(\lambda)v)  ~~
\textbf{e}(0)\rangle  d\lambda\\
&&=\frac{1}{t-s} \int_s^t R(s,\lambda) d\lambda, \eean where
$G(s,\lambda)= \sum_{j\ge 1} \overline{f_j(\lambda)} \langle
\textbf{e}(f),(V_{s,\lambda}-1) (u,L_j(\lambda)v)  ~~ \textbf{e}(0)\rangle. $
 Note that the complex valued function
$R(s,\lambda)$ is locally integrable in $\lambda$ and continuous in $s $
 and  and therefore it makes sense  to talk about $R(s,s)$  which is $0.$ So we get
\[\lim_{t \downarrow s} \langle
\textbf{e}(f),\zeta(s,t,u,v)\rangle=0.\] Since  $\zeta(s,t,u,v)$ is
uniformly bounded in $s,t$
\[\lim_{t \downarrow s} \langle
\phi,\zeta(s,t,u,v)\rangle=0, \forall \phi \in \Gamma.\]

 We also have
 \be \label{xi1}  \langle \phi,
\gamma(s,t,u,v)\rangle = \frac{1}{t-s}\sum_{j\ge 1} \int_s^t \langle
u,L_j(\lambda)v\rangle \overline{\phi^{(1)}_j} (\lambda )
d\lambda.\ee Since
\[|\sum_{j\ge 1} \langle
u,L_j (\lambda ) v\rangle \overline{\phi^{(1)}_j} (\lambda ) |^2\le
\|u\|^2 \sum_{j\ge 1} \|L_j  (\lambda )v\|^2 |\phi^{(1)}_j (\lambda
)|^2\le C_\tau \|v\|^2 \|\phi^{(1)} (\lambda)\|^2,\] the function
$\sum_{j\ge 1} \langle u,L_j  (\lambda ) v\rangle
\overline{\phi^{(1)}_j} (\lambda)$ is in $ L^2$ and hence locally
integrable. Thus we get
\[\lim_{t \downarrow s} \langle \phi,
\gamma(s,t,u,v)\rangle=\sum_{j\ge 1} \langle u,L_j(s)v\rangle
\overline{\phi^{(1)}_j} (s)=\langle \phi^{(1)} (s),
\eta_s(u,v)\rangle ~~\mbox{a.e.}~~ s\ge 0.\]

\end{proof}
\blema \label{Mst}
 For $n\ge 1, ~ \underbar{s}\in \Sigma_n$  and $u_k,v_k\in
\mathbf h:k=1,2,\cdots, n, \phi \in \Gamma(L^2(\mathbb R_+,\mathbf
k))$ and  disjoint $[s_k,t_k),$
\begin{itemize}
\item[(i)] $\lim_{\underbar{t} \downarrow \underbar{s}}\langle \phi,
\prod_{k=1}^n M(s_k,t_k,u_k,v_k)~~\textbf{e}(0) \rangle=0,$ \\
where $M(s_k,t_k,u_k,v_k)=\frac{(V_{s_k,t_k}-1)}{t_k-s_k}(u_k,v_k)-
\rho(s_k,t_k,u_k,v_k)-\gamma(s_k,t_k,u_k,v_k)$  and
$\lim_{\underbar{t} \downarrow \underbar{s}}$  means $t_k\downarrow
s_k$ for each $k.$

\item[(ii)]
$\lim_{\underbar{t} \downarrow \underbar{s}}\langle \phi,
\otimes_{k=1}^n \gamma(s_k,t_k,u_k,v_k) \rangle=\langle
\phi^{(n)}(s_1,s_2,\cdots, s_n), \eta_{s_1}(u_1,v_1) \otimes\cdots
\otimes \eta_{s_n}(u_n,v_n) \rangle.$
\end{itemize}
 \elema
 \begin{proof}
 {(i)} First note that
$M(s,t,u,v)\textbf{e}(0)=\zeta(s,t,u,v)+~ \varsigma(s,t,u,v).$
 So by the above observations in Lemma \ref{lemma-vst-1},
 $\{M(s,t,u,v) \textbf{e}(0)\}$ is  uniformly bounded in $s,t\le \tau$
 and\\
$\lim_{t \downarrow s} \langle \textbf{e}(f)
,M(s,t,u,v)\textbf{e}(0)\rangle=0, \forall f\in L^2(\mathbb
R_+,\mathbf k).$ Since the intervals $[s_k,t_k)$'s are disjoint for
different $k$'s,
\[\langle \textbf{e}(f), \prod_{k=1}^n M(s_k,t_k,u_k,v_k)~~\textbf{e}(0)
\rangle= \prod_{k=1}^n \langle \textbf{e}(f_{[s_k,t_k)}),
M(s_k,t_k,u_k,v_k)~~\textbf{e}(0) \rangle\] and thus
$\lim_{\underbar{t} \downarrow \underbar{s}}\langle \textbf{e}(f),
\prod_{k=1}^n M(s_k,t_k,u_k,v_k)~~\textbf{e}(0) \rangle=0.$\\
 Since  $\prod_{k=1}^n
M(s_k,t_k,u_k,v_k)~~\textbf{e}(0)$ is uniformly bounded  in
$s_k,t_k$
 requirement follows for  $\phi\in \Gamma.$

\noindent {(ii) } It can be proved similarly as part {(iii) } of the
previous Lemma.

\end{proof}

\begin{lemma}
Let $\phi\in \Gamma$ be such that
\begin{equation}\label{ortho}
\left\langle \phi,\psi\right\rangle=0,\qquad \forall~ \psi\in\mathcal{S}^\prime.
\end{equation}
Then we have
\begin{itemize}
\item[(i)] $\phi^{(0)}=0$ and $\phi^{(1)}=0,$
\item[(ii)] for any $n\ge 0$, $\phi^{(n)}=0,$
\item[(iii)] the set $\mathcal{S}^\prime$ is total in the Fock space $\Gamma$.
\end{itemize}
\end{lemma}

\begin{proof}
(i) For any $s\ge 0$, $V_{s,s}=1_{\mathbf{h}\otimes \Gamma}$ and so,
in particular, (\ref{ortho}) gives, for any $u,v\in \mathbf{h}$,
\[
0=\left\langle \phi, V_{s,s}(u,v)  \textbf{e}(0)\right\rangle
=\left\langle u,v\right\rangle
\overline{\phi^{(0)}}
\]
and hence $\phi^{(0)}=0$.

(ii) By \eqref{ortho}, $\left\langle
\phi,\left[V_{s,t}-1\right](u,v)\textbf{e}(0)\right\rangle=0$ for
any $0\le s < t\le \tau<\infty$ and $u,v\in\mathbf{h}$. By HP
equation \eqref{hpeqn,st11} and part (iii) of  Lemma
\ref{lemma-vst-1} , we have
\begin{align*}
0
 &=\lim_{t\downarrow s}\frac{1}{t-s}\left\langle \phi, [V_{s,t}-1] (u,v) \textbf{e}(0)\right\rangle \\
 &=\sum_{j\ge 1}\left\langle u,L_j(s)v \right\rangle \overline{\phi^{(1)}_j(s)}\\
 &=\langle \phi^{(1)} (s),
\eta_s(u,v)\rangle.
\end{align*}
So $\left\langle \phi^{(1)}(s),\eta_s(u,v)\right\rangle=0$ for any
$u,v\in \mathbf{h}$ for almost every  $s$. Since
$\{\eta_s(u,v):u,v\in\mathbf{h}\}$ is total in $\mathbf{k}_s$, it
follows that $\phi^{(1)}(s)=0\in\mathbf{k}_s$ for almost every $0\le
s\le \tau,$  i.e,  $\phi^{(1)}=0.$

(iii) We prove this by induction. The result is already proved for
$n=0,1$. For $n\ge 2$, assume as induction hypothesis that for all
$m\le n-1$, $\phi^{(m)}( \underbar{{\it s}})=0$, for almost every
$\underbar{{\it s}}\in \Sigma_m$ ($s_i\le \tau$ for
$i=1,2,\cdots,m$). To show that $\phi^{(n)}=0,$ we note that by a
similar argument as in \cite{SSS},
\[
\left\langle\phi^{(n)}(s_1,s_2,\cdots, s_n), \eta_{s_1}(u_1,v_1)
\otimes\cdots \otimes \eta_{s_n}(u_n,v_n) \right\rangle=0.
\] for almost every
$\underbar{{\it s}}\in \Sigma_n$ ($s_i\le \tau$).
Since $\{ \eta_s(u,v):u,v\in \mathbf{h}\}$ is total in
$\mathbf{k}_s$, it follows that $\phi^{(n)}(s_1,s_2,\cdots,s_n)=0\in
\mathbf k_{s_1}\otimes\cdots\otimes\mathbf{k}_{s_n}$ for almost
every $(s_1,s_2,\cdots, s_n)\in \Sigma_n$.
\end{proof}
%

\subsection{Unitary Equivalence}

We shall now prove (ii) in Theorem \ref{mainthm}
that the unitary evolution $\{U_{s,t}\}$ on $\mathbf{h} \otimes\mathcal{H}$
is unitarily equivalent to the unitary solution $\{V_{s,t}\}$ of HP equation \eqref{hpeqn,st11}.
To prove this we need the following two results.

\begin{lemma}
Let
$U_{\underbar{s},\underbar{t}}(\underbar{u},\underbar{v})\Omega$ and
$U_{\underbar{s}^\prime,\underbar{t}^\prime}(\underbar{p},\underbar{w})\Omega$
be in $\mathcal{S}$, where
$\underbar{v},\underbar{z}\in\mathbf{h}^{\otimes n}$. Then there
exist an integer $m\ge 1$, $\underbar{a}=(a_1,a_2,\cdots,a_m)$,
$\underbar{b}=(b_1,b_2, \cdots, b_m)$ with $0\le a_1\le b_1\le
\cdots\le a_m\le  b_m< \infty$, partition  $R_1\cup R_2\cup
R_3=\{1,\cdots, m\}$ with $|R_i|=m_i$, family of vectors
$x_{k_l},g_{k_i}\in\mathbf{h}$ and $y_{k_l},h_{k_i}\in\mathbf{h}$
for $l\in R_1 \cup R_2$ and $i\in R_2 \cup R_3$ such that
\begin{align*}
U_{\underbar{s},\underbar{t}}(\underbar{u},\underbar{v})
&=\sum_{\underbar{k}} \prod_{l\in R_1 \cup R_2} U_{a_l,b_l}
(x_{k_l},y_{k_l}), \\
 U_{\underbar{s}^\prime,\underbar{t}^\prime}(\underbar{p},\underbar{w})
&=\sum_{\underbar{k}} \prod_{l\in R_2 \cup R_3} U_{a_l,b_l}
(g_{k_l},h_{k_l}).
\end{align*}
\end{lemma}

\begin{proof}
It follows from the evolution hypothesis of $\{U_{s,t}\}$ that for $r\in[s,t]$ and
a complete orthonormal basis $\{f_j\}\in\mathbf{h}$ we can write
$U_{s,t}(u,v)=\sum_{j\ge 1} U_{s,r}(u,f_j)U_{r,t}(f_j,v)$.
\end{proof}

\begin{remark}\label{V}
Since the family of unitary operators $\{V_{s,t}\}$ on $\mathbf{h}\otimes\Gamma$
enjoy all the properties satisfy by family of unitary operators $\{U_{s,t}\}$
on $\mathbf{h}\otimes\mathcal{H}$, the above lemma also hold if we replace $U_{s,t}$ by $V_{s,t}$.
\end{remark}

\begin{lemma}\label{uvinner}
For $U_{\underbar{s},\underbar{t}}(\underbar{u},\underbar{v})
\Omega,
U_{\underbar{s}^\prime,\underbar{t}^\prime}(\underbar{p},\underbar{w})\Omega\in\mathcal{S}$,
we have
\begin{equation}\label{uvinn}
\left\langle
U_{\underbar{s},\underbar{t}}(\underbar{u},\underbar{v}) \Omega,
U_{\underbar{s}^\prime,\underbar{t}^\prime}(\underbar{p},\underbar{w})\Omega\right\rangle
=\left\langle
V_{\underbar{s},\underbar{t}}(\underbar{u},\underbar{v})\textbf{e}(0),
V_{\underbar{s}^\prime,\underbar{t}^\prime}(\underbar{p},\underbar{w})\textbf{e}(0)
\right\rangle.
\end{equation}
\end{lemma}
\begin{proof}
The proof of (\ref{uvinn}) is very similar to that in  \cite{SSS}.
In fact,  for  $$0\le s\le t<\infty, \left\langle
U_{s,t}(u,v)\Omega,U_{s,t}(p,w)\Omega\right\rangle=\left\langle
p,Z_{s,t}(|w><v|) u \right\rangle $$   while $$
 \left\langle V_{s,t} (u,v)\textbf{e}(0),V_{s,t}(p,w)\textbf{e}(0)\right\rangle=\left\langle p,
 \widetilde{Z}_{s,t}(|w><v|)u\right\rangle$$ but
 $\widetilde{Z}_{s,t}=Z_{s,t}.$
 \end{proof}

Now defining  a map  $\Xi: \mathcal H \rightarrow \Gamma$ by sending
$U_{\underbar{{\it s}},\underbar{{\it t}}}(\underbar{{\it
u}},\underbar{{\it v}}) \Omega \in \mathcal S$  to
$V_{\underbar{{\it s}},\underbar{{\it t}}}(\underbar{{\it
u}},\underbar{{\it v}}) \textbf{e}(0)\in \mathcal S^\prime,$  as in
\cite{SSS}, we can establish   unitary equivalence  of HP evolution
$V_{s,t}$ with the evolution   $U_{s,t}$ we  started with.

\section{Appendix}
Let $X$  be a complex separable Banach space.  Consider the Banach
space
\[
\widehat{X}=L^1(\mathbb R_+, X)=\left\{ f:\mathbb R_+ \rightarrow X
~\mbox{a Lebesgue measurable}, \|f\|:=\int_{\mathbb R_+} \|f(\tau)\|
d\tau< \infty\right\}
\]
and define shift operators $U_t$ on $\widehat{X}$ given by
\[
U_t  f(\tau):=\left\{
                  \begin{array}{ll}
                    0 & \hbox{if}~~ \tau <  t, \\
                    f(\tau-t) & \hbox{if}~~ \tau\ge  t.
                  \end{array}
                \right.
\]
Then for each $t\ge 0, U_t$  is an isometry and $\{U_t\}$ is a
strongly continuous semigroup with generator $P=-\frac{d}{dt}$ with
domain
\[
\mathcal D(P)=\left\{f\in\widehat{X}:  f(0)=0, f
   ~\mbox{is absolutely continuous},  f'\in\widehat{X}\right\}.
\]
The adjoint of this semigroup is given by  $U_t^*f(\tau)=f(\tau+t)$.

Let$\{S_{s,t}: 0\le s\le t<\infty\}$ be an evolution of  contraction
operators in $\B(X)$. With further conditions on $S_{s,t} $ we have
the following result

\begin{theorem} \label{StG}
Let  $\{S_{s,t}: 0\le s\le t<\infty\}$ be a contractive  evolution
in $\B(X)$  such that $ \|S_{s,t}-1\|\le C|t-s|$, where  $C$ is
independent of $s,t.$ Then there exists a Lebesgue measurable
function $G:\mathbb R_+\rightarrow \B(X)$ such that  G is  locally
essentially norm
 bounded  and   $$S_{s,t}-1=\int_s^t S_{s, \tau}  G(\tau)  d \tau.$$
\end{theorem}
\begin{proof}
Consider the family  of operators  $\mathcal S_t$ in $\widehat{X}$
define  by
\[
\mathcal  S_t  f (\tau)=S_{\tau, \tau+t} f(\tau+t)=S_{\tau, \tau+t}
U_t^*f(\tau).
\]
Then $\mathcal{S}_t$  is a contractive strongly continuous semigroup
on $\widehat{X}$. To prove the contractivity, for $t\ge 0$, consider
\[
\|\mathcal  S_t  f \|\le \int_{\mathbb R_+}\|S_{\tau, \tau+t}
   f(\tau+t)\| d\tau=
   \int_{\mathbb R_+}\|
   f(\tau+t)\| d\tau\le  \|f\|.
\]
Also, we have the semigroup property:
\begin{align*}
\mathcal{S}_t\mathcal{S}_s f(\tau)
 &=S_{\tau, \tau+t}(\mathcal S_s f)(\tau+t)
   =S_{\tau, \tau+t}  S_{\tau+t,\tau+t+s} f(\tau+t+s)\\
 &=S_{\tau, \tau+t+s}  f(\tau+t+s)
   =\mathcal{S}_{t+s}f(\tau).
\end{align*}
To prove strong continuity of  $\mathcal S_t$, we first consider,
for $g\in \mathcal C_0^\infty(\mathbb R_+,X)$
\begin{align*}
\|(\mathcal S_{t}-1)g\| &=\int _{\mathbb R_+ }\|(\mathcal
S_{t}-1)g(\tau)\|d\tau
   =\int _{\mathbb R_+ }\|(S_{\tau, \tau+t}U_t^*-1)g(\tau)\| d\tau\\
&\le \int _{\mathbb R_+ }\|(S_{\tau, \tau+t}-1)U_t^*g(\tau)\| d\tau
      +\int _{\mathbb R_+ } \|g(t+\tau)-g(\tau)\| d\tau\\
&=C\|g\|t+\int _{\mathbb R_+ } \|g(t+\tau)-g(\tau)\| d \tau.
\end{align*}
We can use  dominated convergence theorem to take limit of the
second  term as $g$ is compactly supported continuous function and
concluded that $\|(\mathcal S_{t}-1)g\|$ converges to  $0$. The
$\mathcal{S}_t$ is strongly continuous follows from density of
$\mathcal{C}_0^\infty(\mathbb R_+,X)$. So there exists a densely
defined, closed maximally  dissipative operator $\mathcal{G}$ which
is the generator of  the semigroup $\mathcal{S}_t$. Thus we have
\[
\mathcal{G}f=\lim_{t\rightarrow 0} \frac{\mathcal{S}_{t}f-f}{t}
\]
in $\widehat{X}$-norm for each  $f\in \mathcal{D}(\mathcal{G})$ and
hence there exists a sequence $t_n$  tending to  $0$  such that for
almost every $\tau$
\begin{align*}
\mathcal{G}f(\tau)
 &=\lim_{t_n\rightarrow 0}\frac{\mathcal S_{t_n}f(\tau)-f(\tau)}{t_n}\\
 &=\lim_{t_n\rightarrow 0} \frac{S_{\tau,\tau+t_n}
  f(\tau+t_n)-f(\tau)}{t_n} \\
 &=\lim_{t_n\rightarrow 0} \left\{\frac{(S_{\tau,\tau+t_n}-1)
  (f(\tau+t_n)-f(\tau))}{t_n} +\frac{(S_{\tau,\tau+t_n}-1)
  f(\tau)}{t_n} +\frac{f(\tau+t_n)-f(\tau)}{t_n}\right\}.
\end{align*}
Since $\left\|\frac{(S_{\tau,\tau+t_n}-1)}{t_n}\right\|\le C, $ for
$f\in  \mathcal D(\mathcal G)\cap \mathcal D(P),$
$\lim_{t_n\rightarrow 0}\frac{\mathcal
S_{t_n}f(\tau)-f(\tau)}{t_n}=\mathcal Gf (\tau)-Pf(\tau)$ for almost
every $\tau$.  Define  it  to  be  $G(\tau)  f(\tau).$ On the other
hand, for any $t,\beta, \sigma\ge 0$
\[
\frac{1}{\sigma} \int_{t+\beta}^{t+\beta+\sigma}  S_{t, \tau }
d\tau- \frac{1}{\sigma} \int_{t}^{t+\sigma}  S_{t, \tau } d\tau=
 \int_{t}^{t+\beta}  S_{t,\tau} \frac{S_{\tau, \tau+\sigma }-1}{\sigma} d\tau.
\]
Therefore, we have $\lim_{\sigma\rightarrow 0} \int_{t}^{t+\beta}
S_{t,\tau} \frac{S_{\tau, \tau+\sigma }-1}{\sigma} d\tau
=S_{t,t+\beta}-1$ by  continuity. Note that the domain $
\mathcal{D}(\mathcal{G})$  contains $\mathcal{C}_0^\infty(\mathbb
R_+,X).$ Let $g\in\mathcal{C}_0^\infty(\mathbb R_+, X)$ such that
for any $t\in (t_1,t_2), g(t)=x$ for some  $x\in X$, such a  $g\in
\mathcal D(\mathcal G)\cap \mathcal D(P).$ Therefore, for almost
every   $t\in (t_1,t_2),$  \be \label{GG}G(t)x=\lim_{\beta
\rightarrow 0}\frac{S_{t,
 t+\beta}g(t)-g(t)}{\beta}= \lim_{\beta \rightarrow 0}\frac{1}{\beta}\int_{t}^{t+\beta}  S_{t,\tau}(\mathcal
 G g(\tau))
 d\tau=\mathcal Gg(t).\ee
Since we  have    $\frac{S_{t, t+\beta}x-x}{\beta}$ is continuous in
$t,$ in particular measurable  and as a point wise limit of
measurable functions $t\mapsto G(t)x$ is Lebesgue measurable.

To see the  boundedness of $G(t)$, consider the following.  By the
identity (\ref{GG}),  $G(t)$ is defined on whole of the Banach space
$X$. So it is enough to show that $G(t)$ is closable and use the
close graph theorem. Let $u_n$ converge to $0$ such that  $G(t)u_n$
converges to  $v.$ 
Let us consider a sequence of  vectors $g_n\in
\mathcal C_0^\infty(\mathbb R_+, X)$ taking value  $u_n$ in an
interval $(t_1,t_2)$ containing $t$  and converging to $0$  in
$\widehat{X}$.  We can choose $g_n\in \widehat{X}$ such
that $Pg_n$  and $\mathcal{G}g_n$   converges and $\mathcal G  g_n (t)=G(t) u_n$ for each $t\in
(t_1,t_2)$
and hence closability of $\mathcal{G}$ gives that  $\mathcal{G}g_{n_k}$,
for a sub sequence,  converges to  $0$ point wise. Since
$\mathcal{G}g_{n_k} (t)=G(t) u_{n_k}$, the limit $v=0$.  Therefore, $G(t)$ is closable and defined everywhere proving that it is bounded for almost all $t.$  Note that 

\begin{equation}\label{GGt}
G(t)x=\lim_{n \rightarrow\infty }\frac{S_{t,t+t_n}-1}{\beta}x
\end{equation}
and by Assumption  $\left\|\frac{S_{t,t+t_n}-1}{t_n}x\right\|\le
C\|x\|.$  Thus we have
\[
\left\|G(t\right\|\le \liminf_{n\rightarrow \infty }
\left\|\frac{S_{t,t+t_n}-1}{t_n}\right\| \le C.
\]
  \end{proof}

\bibliographystyle{amsplain}

\end{document}